\documentclass[13pt]{amsart}
    \usepackage{amsmath}
    \usepackage{}
    \usepackage{amsfonts}
    \usepackage{amsfonts, latexsym, rawfonts, amsmath, amssymb, amsthm}
    \usepackage[plainpages=false]{hyperref}

    \usepackage{pdfsync}

    \usepackage[a4paper]{geometry}
    \usepackage{esint}
    \usepackage{graphics, graphicx}
    \usepackage{tikz}
    \usepackage{appendix}

    \numberwithin{equation}{section}

    \newcommand{\beq}{\begin{equation}}
    \newcommand{\eeq}{\end{equation}}
    \newcommand{\beqs}{\begin{eqnarray*}}
    \newcommand{\eeqs}{\end{eqnarray*}}
    \newcommand{\beqn}{\begin{eqnarray}}
    \newcommand{\eeqn}{\end{eqnarray}}
    \newcommand{\beqa}{\begin{array}}
    \newcommand{\eeqa}{\end{array}}

    \def\lra{\longrightarrow}

    \def\bc{\begin{center}}
    \def\ec{\end{center}}

    \def\begeq{\begin{equation}}
    \def\endeq{\end{equation}}
    \def\and{\quad{\rm and}\quad}

    \let\lra=\longrightarrow

    \def\mapright\#1{\, \smash{\mathop{\lra}\limits^{\#1}}\, }

    \newtheorem{prop}{Proposition}[section]
    \newtheorem{theo}[prop]{Theorem}
    \newtheorem{lem}[prop]{Lemma}
    
    \newtheorem{cor}[prop]{Corollary}
    \newtheorem{rem}[prop]{Remark}
    
    \newtheorem{defi}[prop]{Definition}
    \newtheorem{conj}[prop]{Conjecture}

\begin{document}

     \title{  K\"ahler-Ricci flow for deformed  complex structures }


    \author{Gang Tian, Liang  Zhang  and Xiaohua $\text{Zhu}^{*}$}

    \address{School of Mathematical Sciences \& BICMR,  Peking
    University,  Beijing 100871,  China.}

  \email{ gtian@math.princeton.edu}
    \email{ xhzhu@math.pku.edu.cn}
    \email{ tensor@pku.edu.cn}
    \thanks {* Partially supported by NSFC Grants 11771019 and BJSF Grants Z180004.}

    \begin{abstract} Let $(M,J_0)$ be a Fano manifold which admits a K\"ahler-Ricci soliton, we analyze the behavior of the K\"ahler-Ricci flow near this soliton
    as we deform the complex structure $J_0$. First, we will
    establish an inequality of Lojasiewicz's type for Perelman's entropy along the K\"ahler-Ricci flow. Then we prove the convergence of K\"ahler-Ricci flow when the complex structure associated to the initial value lies in the kernel $Z$ or negative part of the second variation operator of Perelman's entropy.
    As applications, we solve the Yau-Tian-Donaldson conjecture for the existence of K\"ahler-Ricci solitons in the moduli space of complex structures near $J_0$, and we show that the kernel $Z$ corresponds to the local moduli space of Fano manifolds which are modified $K$-semistable. We also prove an uniqueness theorem for K\"ahler-Ricci solitons.

      \end{abstract}

     \date{}

  \maketitle

  \tableofcontents

  \setcounter{section}{-1}

     \section{Introduction}

 Let $(M, J_0)$ be a Fano manifold which admits  a K\"ahler-Ricci, abbreviated as KR, soliton $\omega_{KS}\in 2\pi c_1(M,  J_0)$. \footnote{We always denote a K\"ahler metric $g$ by its K\"ahler form $\omega_g$.}
It is known that for any initial metric $\omega_0\in 2\pi c_1(M,  J_0)$, the KR flow will evolve $\omega_0$ to a KR soliton $\omega_{KS}'$ smoothly \cite{TZZZ, DeS}.  Moreover, by the uniqueness of KR solitons \cite{TZ00, TZ02}, $\omega_{KS}'=\sigma^*\omega_{KS}$ for some $\sigma\in {\rm Aut}(M, J_0)$. Thus it is a natural question how to extend the above convergence result to K\"ahler manifolds with complex structures near $J_0$. The question is closely related to the existence problem of KR solitons for the local moduli of complex structures at $J_0$ as well as the local moduli space  of KR solitons near $\omega_{KS}$.

When $(M, J_0)$ admits a K\"ahler-Einstein (KE) metric, it has been proved that the KR flow is always convergent to a KE metric in the $C^\infty$-topology for   any initial metric $\omega_0\in 2\pi c_1(M,  J)$ so long as  $J$ is sufficiently close to $J_0$ \cite{ZhT, SW15}. As a consequence, the Mabuchi's K-energy on  $(M, J)$ is bounded from below  \cite[Lemma 7.1]{ZhT} (also see  \cite{To12, CS14}),  so all Fano manifolds $(M, J)$ near $(M, J_0)$ are
$K$-semistable. We refer the reader to \cite{LX14}, \cite{Li17} \cite{LXW19}, \cite[Proposition 4.17]{WZ20} etc. for more general K-semistable Fano manifolds.  However, for a KR soliton $(M, J_0, \omega_{KS})$ which is not KE metric, the flow on $(M,J)$ may not always converge to a KR soliton smoothly even if $J$ is close to $J_0$, as explained in \cite [Remark 6.5]{WZ1} by using Pasquier's example of horospherical variety which is a degeneration of the Grassmannian  manifold
${\rm Gr}_q(2,7)$ \cite{Pas}. Thus we are led to understanding those deformed complex structures for which the KR flow is $C^\infty$-convergent to a KR soliton.

According to the deformation theory \cite{Kod, Kur65}, the local moduli of complex structures at $(M,J_0)$ can be
parameterized by using the C$\overset {\vee}  e$ch cohomology class  $H^1(M,  J_0, \Theta)$, which is the infinitesimal deformation space of complex structures on $M$. In our case of Fano manifolds,
we introduce the $h$-harmonic space $\mathcal{H}^{0, 1}_{h}(M,T^{1, 0}M)$ associated to $H^1(M,  J_0, \Theta)$, where $h$ is a Ricci potential of K\"ahler metric $\omega$ in $2\pi c_1(M, J_0)$ (cf. Section 1). In particular,  $h$ is same as a potential function $\theta$ of soliton  vector field (VF) if $\omega=\omega_{KS}$. Then the local deformation of K\"ahler metrics near $(M,J_0,\omega)$ can be parameterized by $\mathcal U_\epsilon=B(\epsilon)\times C^\infty(M)$ (cf. Section 2),
where $B(\epsilon)$ denotes a small $\epsilon$-ball in $\mathcal{H}^{0, 1}_{h}(M,T^{1, 0}M)$ centered at the origin.
Thus the variation space  $\mathcal U$ of K\"ahler metrics becomes
$$\mathcal U=\mathcal{H}^{0, 1}_{h}(M,T^{1, 0}M)\times  C^\infty(M).$$

By computing the second variation of  Perelman's entropy $\lambda(\cdot)$ at $(M, J_0, \omega_{KS})$ via the parameter space $\mathcal U_\epsilon$,
we are able to get the following product formula  of  the second variation operator (cf. Proposition \ref{product}),
$$\delta^2\lambda(\omega_{KS})=H(\psi, \chi)\,=\,H_1(\psi)\oplus H_2(\chi), ~\forall~ (\psi, \chi)\in \mathcal U,$$
where  $H_2(\chi)$ is always  non-positive \cite{TZ08}.
However, the sign of operator $H_1(\psi)$ is in general not definite according to Pasquier's example mentioned above. Thus the convergence problem  of KR flow for deformed complex structure $J_{\psi_\tau}$  associated to $\psi_\tau\in B(\epsilon)\subset \mathcal{H}^{0, 1}_{\theta}(M,T^{1, 0}M)$ $(\epsilon <<1)$ via the Kuranishi map
 will depend on the sign of $H_1(\psi_\tau)$.

Recall the (normalized) KR flow on a Fano manifold $(M,  J)$,
    \begin{align}\label{kr-flow}
    \frac{\partial \widetilde\omega (t)}{\partial t}=-{\rm Ric}\, (\widetilde\omega(t))+\widetilde\omega(t), ~\widetilde\omega_0\in 2\pi c_1(M,  J).
    \end{align}
 By establishing an inequality of Lojasiewicz type  for  the entropy  $\lambda(\cdot)$ along the flow (\ref {kr-flow}),  we will prove

  \begin{theo}\label{main-theorem-3}Let  $(M, J_0)$ be a Fano manifold which admits a KR soliton $\omega_{KS}\in 2\pi c_1(M, J_0)$ with respect to a holomorphic vector field (HVF)  $X$. Then
  there exists  a small  $\epsilon$ such that  for any $\tau\in B(\epsilon)$ with $\psi_\tau\in {\rm Ker}(H_1)=Z$
 the KR  flow (\ref{kr-flow})
  converges smoothly  to a  KR soliton $(M, J_\infty, \omega_{\infty})$  for  any initial
  metric $\widetilde\omega_0$ in $2\pi c_1(M, J_{\psi_\tau} )$.
 Moreover, $X$ can be lifted to become a soliton VF of $(M, J_\infty, \omega_{\infty})$ and
 \begin{align}\label{identity}\lambda(\omega_{\infty})=\lambda(\omega_{KS}).
 \end{align}
Also, the convergence is of polynomial rate.
\end{theo}

We note that $Z=\mathcal{H}^{0, 1}(M,T^{1, 0}M)$ in case of $\omega_{KS}=\omega_{KE}$ \cite{TZ08}. Thus Theorem \ref{main-theorem-3} generalizes the results in \cite{ZhT, SW15}.   We also note that $\psi_\tau\in Z$ if and only if $X$ can be lifted to a HVF on $(M, J_{\psi_\tau})$ (cf. Corollary \ref{lift}). Moreover, as in the case of $\omega_{KS}=\omega_{KE}$, the set of complex structures associated to the kernel $Z$ corresponds to the local moduli space of Fano manifolds which are modified $K$-semistable \cite{Zhang}.
In particular, we prove the following existence theorem of KR solitons in the deformation space of complex structures.

 \begin{theo}\label{existenceThm}Let  $(M, J_0)$ be a Fano manifold which admits a KR-soliton $\omega_{KS}\in 2\pi c_1(M, J_0)$.
   Then  there exists  a small  $\epsilon$ such that $(M, J_{\psi_\tau})$ admits a KR soliton close  to $\omega_{KS}\in 2\pi c_1(M)$ in the Cheeger-Gromov topology  for  $\tau\in B(\epsilon)$ if and only if the
Fano manifold $(M, J_{\psi_\tau})$ is modified $K$-polystable.

 \end{theo}

Theorem \ref{existenceThm} gives a confirmative answer to the Yau-Tian-Donaldson conjecture  for the existence of KR solitons in the   deformation space of complex structures (cf. \cite{WZZ, DS16, BN}, etc.).  We also mention that Theorem \ref{existenceThm} was  proved by Inoue for the equivariant deformation space of complex structures by  using  the deformation theory \cite{Ino19} (also see Corollary \ref{ino-theorem} and Remark \ref{inoue-Th}).  Inoue's result is a generalization of Sz\'ekelyhidi's for KE metrics \cite{Sz}. Our theorem  gives a more general answer for the  existence of  KR-solitons and determines the soliton VFs.

  By Theorem \ref{main-theorem-3}, we can also prove a uniqueness result for KR-solitons in the closure of the orbit by diffeomorphisms, we refer the reader to Theorem \ref{uniqueness-orbit}.

 Let us say a few of words about how to prove   the  Lojasiewicz inequality for   $\lambda(\cdot)$.
  In fact, we will first derive such an inequality for the restricted    $\nu(\cdot)$  of $\lambda(\cdot)$  on  the parameter space $\mathcal U_\epsilon$  (cf. Proposition \ref{lojasiewsz}).   The advantage of $\nu(\cdot)$ is that its gradient and the second variation $\delta^2\nu(\cdot)$  of $\nu(\cdot)$ are both of maps to $\mathcal U$  (cf. Definition  \ref{first-variation-mu}, \ref{second-H}). In particular, we get an explicit formula for the kernel of  $\delta^2\nu(\cdot)$ at a KR soliton (cf. (\ref{kernel-H}) and Remark \ref{kernel}).
  Then by the spectral theorem \cite{Rudin}, we  prove the  Lojasiewicz  inequality for the original  $\lambda(\cdot)$ (cf. Lemma \ref{spectral-application-1}, Corollary \ref{loj-inequ-ZhZ},  (\ref{lajw-inequ-3})).\footnote{ This also answers a  question   of  Chen-Sun \cite[Remark 5.5]{CS14}.}   We would like to mention that such an   inequality  for the space of Riemannian metrics along the Ricci flow has been studied by  Sun-Wang  \cite{SW15}.

 Theorem \ref{main-theorem-3} will be  generalized for  K\"ahler manifolds $(M,J_{\psi_\tau})$ with $H_1(\psi_\tau)\le 0$ as follows.

\begin{theo}\label{main-theorem-tivial-action}Let  $(M, J_0)$ be a Fano manifold which admits a KR soliton $\omega_{KS}\in 2\pi c_1(M, J_0)$.
Suppose that
\begin{enumerate}\label{trivial-action}
  \item  $(M,J_{\psi_\tau})$ admits a KR soliton  for any $ \psi_\tau\in B(\epsilon)\cap Z$;
  \item  $H_1(\cdot)\leq 0, ~{\rm on}~\mathcal{H}^{0, 1}_{\theta}(M,T^{1, 0}M).$
\end{enumerate}
 Then
 there exists    a small  $\epsilon$ such that for any $\tau\in B(\epsilon)$   the KR  flow (\ref{kr-flow})
  converges smoothly  to a  KR soliton $(M, J_\infty, \omega_{\infty})$  for  any initial  metric $\widetilde\omega_0$ in $2\pi c_1(M, J_{\psi_\tau} )$.
  Moreover,
  \begin{align}\lambda(\omega_{\infty})=\lambda(\omega_{KS}), \notag
  \end{align}
   and the convergence is of polynomial rate.
 \end{theo}

 The condition (1)
 may not be necessary in view of Theorem \ref{main-theorem-3}. In fact, we have the following conjecture.

\begin{conj}\label{main-theorem-tivial-action-conjecture}Let  $(M, J_0)$ be a Fano manifold which admits a KR soliton $\omega_{KS}\in 2\pi c_1(M, J_0)$. Let $A_{-}$ be the linear subspace of $\mathcal{H}^{0, 1}_{\theta}(M,T^{1, 0}M)$ associated to the negative eigenvalues of $H_1$.
Then there exists  a small  $\epsilon$ such that  the following is true:
\begin{enumerate}\label{trivial-action}
  \item  For any $\tau\in B(\epsilon)$ with $\psi_\tau\in A_{-}\cup Z$
the   KR  flow (\ref{kr-flow})
  converges smoothly  to a  KR soliton $(M, J_\infty, \omega_{\infty})$  for  any initial  metric $\omega_0$ in $2\pi c_1(M, J_{\psi_\tau} )$.
  Moreover,  the convergence is of polynomial rate with
  $\lambda(\omega_{\infty})=\lambda(\omega_{KS})$.
    \item   For any $\tau\in B(\epsilon)$ with $\psi_\tau\in\mathcal{H}^{0, 1}_{\theta}(M,T^{1, 0}M)\setminus  (A_{-}\cup Z)$, the   KR  flow (\ref{kr-flow}) converges to a singular KR soliton $\omega_\infty$  for  any initial  metric $\omega_0$ in $2\pi c_1(M, J_{\psi_\tau} )$ with $\lambda(\omega_{\infty})$ strictly  bigger than
 $\lambda(\omega_{KS})$.
\end{enumerate}

 \end{conj}

The above conjecture means  that the set of complex structures,  for which the  KR  flow (\ref{kr-flow}) is stable,
  corresponds  to   the  linear semistable  subspace  $A_{-}\cup Z$ of  $\mathcal{H}^{0, 1}_{\theta}(M,T^{1, 0}M)$
in the deformation space of Fano manifolds which admit a KR-soliton.
   Otherwise, the flow on $(M,J_{\psi_\tau})$ is unstable, but  will  in general  converge to a singular KR soliton $\omega_\infty$ with $\lambda(\omega_{\infty})$ strictly  bigger than
 $\lambda(\omega_{KS})$ by the Hamilton-Tian conjecture \cite{Ti97, Ba,  CW,  WZ20}.

  The organization of the paper is as follows. In Section 1, we introduce the weighted Hodge-Laplace operator  $\Box_{h}$ and the $h$-harmonic space of $\mathcal{H}^{0, 1}_{h}(M,T^{1, 0}M)$ on a Fano manifold.  In Section 2,  we  study the  Kuranishi deformation theory of complex structures associated to  $\Box_{h}$ and prove that the solution of modified Kuranishi's equation  is divergent-free  (cf. Proposition\ref{divergence-free-form})). In Section 3,  we  compute the first and second  variations of $\nu(\cdot)$ on  $\mathcal U_\epsilon$. In Section 4,  we  give a new version of the second  variations of $\nu(\cdot)$ and  prove the local maximality of $\lambda(\cdot)$  (cf. Theorem \ref{new-version-second-variation} and Proposition \ref{lml}).  Section 5 is devoted to   prove the Lojasiewicz  inequality (cf. Proposition \ref{lojasiewsz}).  Theorem \ref{main-theorem-3} is  proved in Section 6 while   Both of   Theorem \ref{existenceThm}  and  Theorem \ref{uniqueness-orbit} are proved in Section 7. In  Section 8,  we prove
  Theorem \ref{main-theorem-tivial-action}.

 \section{ Weighted  Hodge-Laplace on $A^{p,q}(M,T^{1,0}M)$ }

 In this section,  we introduce  a  weighted  Hodge-Laplace operator  on the space $A^{p,q+1}(M,T^{1,0}M)$ on a Fano manifold $(M, J_0)$ with a K\"ahler form $\omega\in 2\pi c_1(M)$.  This  is very similar with  the  weighted  Laplace operator associated to the Bakry-\'Emery Ricci curvature studied in \cite{WZ19}.

    Let  $h$ be a Ricci potential of $\omega$, which   satisfies that
 \begin{align}\label{ricci-potential}\rm {Ric}(\omega)-\omega=\sqrt{-1}\partial\bar \partial h.
 \end{align}
 We denote $\delta_{h}=\overline{\partial}_{h}^*$ the dual operator  of $\overline{\partial}$   with respect to the inner product
 $$\int_M<, >e^{h}\omega^n,$$
 where  $<, >=<, >_\omega$ is  the inner product on $\Lambda^{p,q}T^*M\otimes T^{1,0}M$   induced by $\omega$.
 Namely, for any $\varphi\in A^{p,q}(M,T^{1,0}M)$, $\varphi'\in  A^{p,q+1}(M,T^{1,0}M)$, we have
 $$\int_M<\bar\partial \varphi, \varphi'>e^{h}\omega^n= \int_M<\varphi,  \delta_h \varphi'>e^{h}\omega^n.$$
 Then weighted  Hodge-Laplace   operator  $\Box_{h}$ on $ A^{p,q+1}(M,T^{1,0}M) $ is defined by
 $$\Box_{h}=\overline{\partial}\cdot \delta_h+\delta_h\cdot \overline{\partial}.$$
 Thus, similarly  with the  Hodge-Laplace operator, we see that
 $$\Box_{h}\varphi=0$$
 if only if
 $$\bar\partial\varphi=0 ~{\rm and} \ \delta_h\varphi=0.$$
 For simplicity, we denote
 $$\mathcal{H}^{p, q}_{h}(M,T^{1, 0}M)=\{\varphi\in A^{p,q}(M,T^{1,0}M)|~\Box_{h}\varphi=0\}$$
 the harmonic space of  $\Box_{h}$ on $ A^{p,q+1}(M,T^{1,0}M) $.

 In this paper, we are interested in the space $A^{0,1}(M,T^{1,0}M)$.
 Then under local holomorphic coordinates, for any
  $$\varphi=\varphi^i_{\overline{j}}d\overline{z}^j\otimes \frac{\partial}{\partial z^i}\in A^{0,1}(M,T^{1,0}M),$$
 we have
 \begin{align}{\label{del}}\delta_{h}\varphi=-(\varphi^i_{\overline{j},k}+\varphi^i_{\overline{j}}h_k)g^{k\overline{j}}\frac{\partial}{\partial z^i}.
 \end{align}
 Moreover,
  \begin{align}\label{h-harmonic} \mathcal{H}^{0, 1}_{h}(M,T^{1, 0}M)  \cong H^1(M, J, \Theta).
  \end{align}
 The latter is the C$\overset {\vee}  e$ch cohomology group associated to the infinitesimal
 deformation of complex structures on $(M,J_0)$ \cite{Kod}.

 We define  a $h$-divergence operator  on $A^{p,q+1}(M,T^{1,0}M)$  by
 $${\rm div}_h\varphi=e^{-h}{\rm div}(e^{h}\varphi).$$
 In particular, if
 $$\varphi=\varphi^i_{\overline{j}}d\bar z^j \otimes\frac{\partial}{\partial z^i} \in A^{0,1}(M,T^{1,0}M), $$
 then
 \begin{align}{\label{div}}
 {\rm div}_h\varphi=(\varphi^{i}_{\overline{j},i}+\varphi^{i}_{\overline{j}}h_i)d\bar z^{j}.
 \end{align}
Also  we introduce an  inner product
 $\varphi\lrcorner\mu$ with respect to
  $$\mu=\sqrt{-1}\mu_{k\bar l}  dz^k\wedge d\bar z^l\in A^{1,1}(M) $$
    by
 $$\varphi\lrcorner\mu=\sqrt{-1}\varphi^k_{\overline{j}} \mu_{k\overline{l}} d\bar z^j\wedge d\bar z^l \in A^{0,2}(M).$$

 We list some  identities for the above two operators in the following lemma.

 \begin{lem}\label{KRS cond}
   \begin{enumerate}
 \item $\overline{\partial}(\varphi\lrcorner\mu)=\overline{\partial}\varphi\lrcorner\mu+\varphi\lrcorner\overline{\partial}\mu.$
     \item $\delta_h({\rm div}_{h}\varphi)={\rm div}_{h}(\delta_h \varphi).$
     \item $\delta_h(\varphi\lrcorner \omega)=-\delta_h\varphi\lrcorner \omega-\sqrt{-1}{\rm div}_{h}\varphi.$
     \item$\overline{\partial}{\rm div}_{h}\varphi-{\rm div}_{h}(\overline{\partial}\varphi)=-\sqrt{-1}\varphi\lrcorner \omega.$
     \item ${\rm div}_h[\varphi, \psi]=\psi\lrcorner\partial {\rm div}_h \varphi+\varphi\lrcorner\partial {\rm div}_h\psi.$
     \item $[\varphi, \psi]\lrcorner \omega=\varphi\lrcorner\partial(\psi\lrcorner \omega)+\psi\lrcorner\partial(\varphi\lrcorner \omega).$
   \end{enumerate}
 \end{lem}

 \begin{proof}All  identities can be verified directly.
 $(1)$ is simple. For $(2)$, we note that
 $$\delta_h({\rm div}_{h}\varphi)=-(\varphi^{i}_{\overline{j},is}+\varphi^{i}_{\overline{j}s}h_{i}+\varphi^{i}_{\overline{j}}h_{is}+(\varphi^{i}_{\overline{j},i}+\varphi^{i}_{\overline{j}}h_i)h_s)g^{s\overline{j}} $$
 and
   $${\rm div}_{h}(\delta_h \varphi)=-(\varphi^{i}_{\overline{j},si}+\varphi^{i}_{\overline{j}i}h_s+\varphi^{i}_{\overline{j}}h_{si}+(\varphi^{i}_{\overline{j},s}+\varphi^{i}_{\overline{j}}h_s)h_i)g^{s\overline{j}}.$$
 Thus $(2)$ is true.

 Since $\varphi\lrcorner\omega=\sqrt{-1}\varphi^{p}_{\overline{k}}g_{p\overline{j}}d\overline{z}^{k}\wedge d\overline{z}^{j}$, we have
 \begin{align*}
     \delta_h (\varphi\lrcorner\omega) & =-\sqrt{-1}(\varphi^{p}_{\overline{k},s}g_{p\overline{j}}+\varphi^{p}_{\overline{k}}g_{p\overline{j}}h_s)g^{s\overline{j}}d\overline{z}^k+
     \sqrt{-1}(\varphi^{p}_{\overline{k},s}g_{p\overline{j}}+\varphi^{p}_{\overline{k}}g_{p\overline{j}}h_s)g^{s\overline{k}}d\overline{z}^j \\
     &=-\sqrt{-1}(\varphi^{s}_{\overline{k},s}+\varphi^{s}_{\overline{k}}h_s)d\overline{z}^k-\sqrt{-1}(\delta_h\varphi)^pg_{p\overline{j}}d\overline{z}^j\\
     &=-\sqrt{-1}{\rm div}_{h}\varphi-(\delta_h\varphi\lrcorner\omega).
   \end{align*}
 This is  $(3)$.

 In order to prove $(4)$, we choose a normal coordinate $(z^1,\ldots,z^n)$ around each  $p\in M$. By (\ref{div}), we have
   $$\overline{\partial}{\rm div}_{h}\varphi=(\varphi^s_{\overline{k},s\overline{l}}+\varphi^s_{\overline{k},\overline{l}}h_s+\varphi^s_{\overline{k}}h_{s\overline{l}})d\overline{z}^l\wedge d\overline{z}^k.$$
 Since
   $$\overline{\partial}\varphi=(\varphi^s_{\overline{k},\overline{l}}+\Gamma^{\overline{h}}_{\overline{k}\overline{l}}\varphi^s_{\overline{h}})d\overline{z}^l\wedge d\overline{z}^k\otimes \frac{\partial}{\partial z^s},$$
 we see that  at $p$,
   $${\rm div}_{h}\overline{\partial}\varphi=(\varphi^s_{\overline{k},\overline{l}s}+\partial_s\Gamma^{\overline{h}}_{\overline{k}\overline{l}}\varphi^s_{\overline{h}}+\varphi^{s}_{\overline{k},\overline{l}}h_s)d\overline{z}^l\wedge d\overline{z}^k.$$
 Thus we get
 \begin{align}\label{1.1(4)}\overline{\partial}{\rm div}_{h}\varphi-{\rm div}_{h}(\overline{\partial}\varphi)=(\varphi^s_{\overline{k},s\overline{l}}-\varphi^s_{\overline{k},\overline{l}s}+\varphi^s_{\overline{k}}h_{s\overline{l}}-\partial_s\Gamma^{\overline{h}}_{\overline{k}\overline{l}}\varphi^s_{\overline{h}})d\overline{z}^l\wedge d\overline{z}^k.
 \end{align}

 On the other hand,
   \begin{align*}
     \varphi^s_{\overline{k},s\overline{l}}-\varphi^s_{\overline{k},\overline{l}s}&=-R_{h\overline{t}s\overline{l}}g^{s\overline{t}}\varphi^h_{\overline{k}}-
  R_{t\overline{k}s\overline{l}}g^{t\overline{h}}\varphi^s_{\overline{h}} \\
 &=-R_{h\overline{l}}\varphi^{h}_{\overline{k}}+\partial_s\Gamma^{\overline{h}}_{\overline{k}\overline{l}}\varphi^s_{\overline{h}}.
   \end{align*}
 By (\ref{1.1(4)}) it follows that
 $$\overline{\partial}{\rm div}_{h}\varphi-{\rm div}_{h}(\overline{\partial}\varphi)=(-R_{s\overline{l}}+h_{s\overline{l}})\varphi^s_{\overline{k}}d\overline{z}^l\wedge d\overline{z}^k.$$
 We note that
 $$(-R_{s\overline{l}}+h_{s\overline{l}})\varphi^s_{\overline{k}}d\overline{z}^l\wedge d\overline{z}^k=\varphi\lrcorner\frac{{\rm Ric}(\omega)-\sqrt{-1}\partial\overline{\partial}h}{\sqrt{-1}}.$$
 Together with the assumption ${\rm {Ric}}(\omega)-\omega=\sqrt{-1}\partial\bar \partial h$, we conclude that
 $$\overline{\partial}{\rm div}_{h}\varphi-{\rm div}_{h}(\overline{\partial}\varphi)=\varphi\lrcorner\frac{{\rm Ric}(\omega)-\sqrt{-1}\partial\overline{\partial}h}{\sqrt{-1}}=-\sqrt{-1}\varphi\lrcorner\omega,$$
 which proves $(4)$.

 For $(5)$, we recall  that for any $\varphi=\varphi^{l}_{\overline{k}}d\overline{z}^k\otimes\frac{\partial}{\partial z^l}$ and $\psi=\psi^{j}_{\overline{i}}d\overline{z}^i\otimes\frac{\partial}{\partial z^j}$ it holds
 \begin{align*}
  [\varphi,\psi]&=\varphi^{l}_{\overline{k}}\frac{\partial \psi^{j}_{\overline{i}}}{\partial z^l}d\overline{z}^k\wedge d\overline{z}^i\otimes\frac{\partial}{\partial z^j}+\psi^{j}_{\overline{i}}\frac{\partial \varphi^{l}_{\overline{k}}}{\partial z^j}d\overline{z}^i\wedge d\overline{z}^k\otimes\frac{\partial}{\partial z^l}\\
  &=\varphi^{l}_{\overline{k}}\frac{\partial \psi^{j}_{\overline{i}}}{\partial z^l}d\overline{z}^k\wedge d\overline{z}^i\otimes\frac{\partial}{\partial z^j}-\psi^{l}_{\overline{i}}\frac{\partial \varphi^{j}_{\overline{k}}}{\partial z^l}d\overline{z}^k\wedge d\overline{z}^i\otimes\frac{\partial}{\partial z^j}\\
  &=(\varphi^{l}_{\overline{k}}\frac{\partial \psi^{j}_{\overline{i}}}{\partial z^l}-\psi^{l}_{\overline{i}}\frac{\partial \varphi^{j}_{\overline{k}}}{\partial z^l})d\overline{z}^k\wedge d\overline{z}^i\otimes\frac{\partial}{\partial z^j}\\
  &=(\varphi^{l}_{\overline{k}}\psi^{j}_{\overline{i},l}-\psi^{l}_{\overline{i}}\varphi^{j}_{\overline{k},l}-(\varphi^{l}_{\overline{k}}\Gamma^j_{lh}\psi^h_{\overline{i}}-\psi^{l}_{\overline{i}}\Gamma^j_{hl}\varphi^h_{\overline{k}}))d\overline{z}^k\wedge d\overline{z}^i\otimes\frac{\partial}{\partial z^j}\\
  &=(\varphi^{l}_{\overline{k}}\psi^{j}_{\overline{i},l}-\psi^{l}_{\overline{i}}\varphi^{j}_{\overline{k},l})d\overline{z}^k\wedge d\overline{z}^i\otimes\frac{\partial}{\partial z^j}.
 \end{align*}
 Then,
 \begin{align*}
   {\rm div}[\varphi,\psi]&=(\varphi^{l}_{\overline{k}}\psi^{j}_{\overline{i},l})_jd\overline{z}^k\wedge d\overline{z}^i-(\psi^{l}_{\overline{i}}\varphi^{j}_{\overline{k},l})_jd\overline{z}^k\wedge d\overline{z}^i \\
  &=(\varphi^{l}_{\overline{k},j}\psi^{j}_{\overline{i},l}+\varphi^{l}_{\overline{k}}\psi^{j}_{\overline{i},lj}-\psi^{l}_{\overline{i},j}\varphi^{j}_{\overline{k},l}-\psi^{l}_{\overline{i}}\varphi^{j}_{\overline{k},lj})d\overline{z}^k\wedge d\overline{z}^i\\
   &=(\varphi^{j}_{\overline{k},l}\psi^{l}_{\overline{i},j}+\varphi^{l}_{\overline{k}}\psi^{j}_{\overline{i},lj}-\psi^{l}_{\overline{i},j}\varphi^{j}_{\overline{k},l}-\psi^{l}_{\overline{i}}\varphi^{j}_{\overline{k},lj})d\overline{z}^k\wedge d\overline{z}^i\\
   &=(\varphi^{l}_{\overline{k}}\psi^{j}_{\overline{i},lj}-\psi^{l}_{\overline{i}}\varphi^{j}_{\overline{k},lj})d\overline{z}^k\wedge d\overline{z}^i.
 \end{align*}
 Hence by (\ref{div}), we get
 \begin{align}\label{(1.1)(5)(1)}
   {\rm div}_{h}[\varphi,\psi]&=(\varphi^{l}_{\overline{k}}\psi^{j}_{\overline{i},lj}-\psi^{l}_{\overline{i}}\varphi^{j}_{\overline{k},lj}+\varphi^{l}_{\overline{k}}\psi^{j}_{\overline{i},l}h_j-\psi^{l}_{\overline{i}}\varphi^{j}_{\overline{k},l}h_j)d\overline{z}^k\wedge d\overline{z}^i.
 \end{align}

 On the other hand, by (\ref{div}) we also have
 $$\partial {\rm div}_{h}\varphi=(\varphi^l_{\overline{k},lj}+\varphi^l_{\overline{k},j}h_l+\varphi^l_{\overline{k}}h_{lj})dz^j\wedge d\overline{z}^k.$$
 It follows that
 \begin{align}\label{(1.1)(5)(2)}
   \psi\lrcorner\partial {\rm div}_{h}\varphi=(\psi^{l}_{\overline{i}}\varphi^{j}_{\overline{k},lj}+\psi^l_{\overline{i}}\varphi^j_{\overline{k},l}h_j+\psi^l_{\overline{i}}\varphi^j_{\overline{k}}h_{jl})d\overline{z}^i\wedge d\overline{z}^k.
 \end{align}
 Similarly we have
 \begin{align}\label{(1.1)(5)(3)}
   \varphi\lrcorner\partial {\rm div}_{h}\psi=(\varphi^{l}_{\overline{i}}\psi^{j}_{\overline{k},lj}+\varphi^l_{\overline{i}}\psi^j_{\overline{k},l}h_j+\varphi^l_{\overline{i}}\psi^j_{\overline{k}}h_{jl})d\overline{z}^i\wedge d\overline{z}^k.
 \end{align}
 We note that $h_{ij}=h_{ji}$.  Hence,  combing  the above (\ref{(1.1)(5)(1)})-(\ref{(1.1)(5)(3)})  we will derive $(5)$ immediately.

 Finally we prove $(6)$. From the proof of $(5)$ we see that
 $$[\varphi,\psi]=(\varphi^{l}_{\overline{k}}\psi^{j}_{\overline{i},l}-\psi^{l}_{\overline{i}}\varphi^{j}_{\overline{k},l})d\overline{z}^k\wedge d\overline{z}^i\otimes\frac{\partial}{\partial z^j}.$$
 Thus we have
 \begin{align}\label{(1,1)(6)(1)}[\varphi,\psi]\lrcorner\omega=\sqrt{-1}\varphi^{l}_{\overline{k}}\psi^{j}_{\overline{i},l}g_{j\overline{t}}d\overline{z}^k\wedge d\overline{z}^i\wedge d\overline{z}^t-\sqrt{-1}\psi^{l}_{\overline{i}}\varphi^{j}_{\overline{k},l}g_{j\overline{t}}d\overline{z}^k\wedge d\overline{z}^i\wedge d\overline{z}^t.
 \end{align}

 On the other hand, we have
 \begin{align*}
 \partial(\psi\lrcorner\omega)=\sqrt{-1}\psi_{\overline{i},t}^jg_{j\overline{k}}dz^t\wedge d\overline{z}^i\wedge d\overline{z}^k.
 \end{align*}
 It follows that
 \begin{align}\label{(1,1)(6)(2)}
   \varphi\lrcorner\partial(\psi\lrcorner\omega)=\sqrt{-1}\varphi^l_{\overline{k}}\psi_{\overline{i},l}^jg_{j\overline{t}}d\overline{z}^i\wedge d\overline{z}^t\wedge d\overline{z}^k.
 \end{align}
 Similarly we have
 \begin{align}\label{(1,1)(6)(3)}
 \psi\lrcorner\partial(\varphi\lrcorner\omega)=\sqrt{-1}\psi^l_{\overline{i}}\varphi_{\overline{k},l}^jg_{j\overline{t}}d\overline{z}^k\wedge d\overline{z}^t\wedge d\overline{z}^i.
 \end{align}
 Hence, combining  the above (\ref{(1,1)(6)(1)})-(\ref{(1,1)(6)(3)}) we get $(6)$.
 \end{proof}

  \begin{lem}\label{tech}
 Suppose that
 $$\delta_h \varphi=0~{\rm  and}~ \overline{\partial}(\varphi\lrcorner \omega)=0.$$
  Then
 \begin{align}\label{tech2}
 \Box_{h}(\varphi\lrcorner\omega)=-\sqrt{-1}({\rm div}_{h}\overline{\partial}\varphi)-\varphi\lrcorner\omega.
 \end{align}
 \end{lem}

 \begin{proof}

 Using the assumption $\overline{\partial}(\varphi\lrcorner \omega)=0$, we have
   \begin{align*}
     \Box_{h}(\varphi\lrcorner\omega)&=\overline{\partial}\delta_h(\varphi\lrcorner\omega)+\delta_h\overline{\partial}(\varphi\lrcorner\omega)\\
 &=\overline{\partial}\delta_h(\varphi\lrcorner\omega).
 \end{align*}
 Then by Lemma \ref{KRS cond}-$(3)$ and the assumption $\delta_h \varphi=0$, we get
 \begin{align*}
 \Box_{h}(\varphi\lrcorner\omega)&=\overline{\partial}(-\delta_h\varphi\lrcorner\omega-\sqrt{-1}{\rm div}_{h}\varphi)\\
 &=\overline{\partial}(-\sqrt{-1}{\rm div}_{h}\varphi).
 \end{align*}
 Hence, by Lemma \ref{KRS cond}-$(4)$, we obtain  (\ref{tech2}).
 \end{proof}

 \begin{cor}\label{tech1} Let $\varphi\in \mathcal{H}^{0, 1}_{h}(M,T^{1, 0}M).$ Then
  $$\Box_{h}(\varphi\lrcorner \omega)=-\varphi\lrcorner\omega.$$
   As a consequence,
   \begin{align}\label{compatiable-2}\varphi\lrcorner\omega=0.
   \end{align}
   \end{cor}

 \begin{proof}
 Note that  $\overline{\partial}\varphi=0$. Then  by Lemma \ref{KRS cond}-$(1)$ we have $\overline{\partial}(\varphi\lrcorner \omega)=0$.  On the other hand,  since  $\delta_h\varphi=0$,   by  Lemma \ref{tech} together with $\overline{\partial}\varphi=0$, we see that
 \begin{align*}
     \Box_{h}(\varphi\lrcorner\omega)&=-\sqrt{-1}{\rm div}_{h}(\overline{\partial}\varphi)-\varphi\lrcorner \omega\\
 &=-\varphi\lrcorner \omega.
   \end{align*}
 Thus
 \begin{align*}
 0\leq(\varphi\lrcorner\omega,\varphi\lrcorner\omega)_{h}&=-(\Box_{h}(\varphi\lrcorner\omega),\varphi\lrcorner\omega)_h\\
 &=-(\delta_h(\varphi\lrcorner\omega),\delta_h(\varphi\lrcorner\omega))_h-(\overline{\partial}(\varphi\lrcorner\omega),\overline{\partial}(\varphi\lrcorner\omega))_h\\
 &\leq 0.
 \end{align*}
 It follows that
 $$(\varphi\lrcorner\omega,\varphi\lrcorner\omega)_{h}=0,$$
 which implies (\ref{compatiable-2}).

 \end{proof}

 \section{ Kuranishi deformation theory for  $\Box_{h}$-operator }

 We choose a basis $\{e_i\}_{i=1,..., l}$ of  $\mathcal{H}^{0, 1}_{h}(M,T^{1, 0}M)$ and denote an $\epsilon$-ball by
 $$B(\epsilon)=\{ \tau=(\tau_1,...,\tau_l)|~  \psi_\tau=\sum \tau_i e_i\in \mathcal{H}^{0, 1}_{h}(M,T^{1, 0}M) ~{\rm with}~ \sum_i \tau_i^2 < \epsilon^2\}.$$
 Then by  the Kodaira deformation  theory for  complex structures on $(M,J_0)$,
 there is  a map
 $$\Psi:  B(\epsilon)\mapsto A^{0, 1}(M,T^{1, 0}M)$$
 as long as $\epsilon$ is small enough such that $\Psi(t)$ is a  solution of Cartan-Maurer equation,
 \begin{align}\label{CM-equation}\overline{\partial}\Psi(\tau)=\frac12[\Psi(\tau), \Psi(\tau)].
 \end{align}
 More precisely,   analogous to  the Kodaira theory \cite{Kod},   we can  reduce   (\ref{CM-equation}) to solving the following  equation with gauge fixed,
   \begin{align}\label{kodaira-equation}\left\{
     \begin{array}{ll}
       \delta_{h}\varphi(\tau)=0\\
       \overline{\partial}\varphi(\tau)=\frac12[\varphi(\tau), \varphi(\tau)]. \\
     \end{array}
   \right.
   \end{align}
  The above equation is equivalent to
 \begin{align}\label{kuranish-equation}\varphi(\tau)=\sum^{m}_{i=1}\tau_ie_i+\frac12\delta_{h}G_h[\varphi(\tau), \varphi(\tau)],
 \end{align}
 where $G_h$ the Green operator associated to $\Box_h$.  Following the Kuranshi's method by the implicity function theorem \cite{Kur65}, there is a unique solution of (\ref{kuranish-equation}) as long as $\sum |\tau_i|^2<\epsilon^2$.  Thus there is a  map
 \begin{align}\label{epsilon-choice} \Phi: B(\epsilon)\mapsto A^{0, 1}(M,T^{1, 0}M)
 \end{align}
 such that
 $\Phi (\tau)=\varphi(\tau)=\varphi_\tau$. For simplicity, we call  $\Phi$ the Kuranishi map associated to  $\Box_{h}$-operator.

 On the other hand,  according to  \cite{Kod},  we may  write $\varphi(\tau)$ as a convergent expansion to solve  (\ref{CM-equation}),
 \begin{align}\label{expansion-0}
 \varphi(\tau)=\sum \tau_ie_i +\sum_{|I|\le 2} \tau^I\varphi_I.
 \end{align}
By (\ref{expansion-0}), we prove

 \begin{prop}\label{divergence-free-form}
   Let $\varphi(\tau)$ be a Kuranishi solution of  (\ref{kuranish-equation})  as the form of  (\ref{expansion-0}).
   Then
    \begin{align}\label{compatiable-divergence-free}\varphi(\tau)\lrcorner\omega=0 ~{\rm and}~
    {\rm div}_{h}\varphi(\tau)=0.
  \end{align}
 \end{prop}

 \begin{proof}
 Let
 \begin{align}\label{psi-k}\psi_k=\sum_{|I|\leq k}\tau^{I}\varphi_{I}, k=1, 2, \ldots.
\end{align}
 Then
 $$\delta_{h}\psi_k=0, ~{\rm for ~each} ~k\ge 1.$$
  We need to prove (\ref{compatiable-divergence-free}) for each $\psi_k$  by induction.  Since
 $$ \delta_{h}\psi_1=0, \overline{\partial}\psi_1=0, $$
    by Corollary \ref{tech1},  we have
 $$\psi_1\lrcorner\omega=0.$$
  By Lemma \ref{KRS cond}-$(3)$,  it follows that
 \begin{align}\label{divergence-free}{\rm div}_h\psi_1=0.
 \end{align}
 Thus we may assume  that
 \begin{align}\label{assumption-2}\psi_k\lrcorner \omega=0 ~ {\rm and}~  {\rm div}_h\psi_k=0, ~ {\rm for ~any }~ k< l.
 \end{align}

 By (\ref{kodaira-equation}), we see that
 \begin{align}\label{condition-1}\overline{\partial}\psi_l=\frac12\sum_{k=1}^{l-1}[\varphi_k, \varphi_{l-k}], \delta_{h}\psi_l=0.
 \end{align}
 Then by  Lemma \ref{KRS cond}-(6)  together with the first relation in (\ref{assumption-2}),  we get
 $$\overline{\partial}\psi_l\lrcorner\omega=0.$$
 By  Lemma \ref{KRS cond}-(1), it follows that
 \begin{align}\label{bar-trace-free}\overline{\partial}(\psi_l\lrcorner\omega)=0.
 \end{align}
 Thus by  Lemma \ref{tech} together with the second relation in (\ref{condition-1}),  we derive
 \begin{align}\label{box-equation}\Box_{h}(\psi_l\lrcorner\omega)=-\sqrt{-1} {\rm div}_{h}\overline{\partial}\psi_l-\psi_l\lrcorner\omega.
 \end{align}

 On the other hand,
 \begin{align*}
   {\rm div}_{h}\overline{\partial}\psi_l &= {\rm div}_{h}(\frac12\sum_{k=1}^{l-1}[\varphi_k, \varphi_{l-k}]) \\
 &=\frac12\sum_{k=1}^{l-1} {\rm div}_{h}[\varphi_k, \varphi_{l-k}].
 \end{align*}
 By  Lemma \ref{KRS cond}-(5), we have
  $${\rm div}_{h}\overline{\partial}\psi_l=0.$$
 Thus by  (\ref{box-equation}),   we get
 $$\Box_{h}(\psi_l\lrcorner\omega)=-\psi_l\lrcorner\omega.$$
  Hence,  as in the proof of  (\ref{compatiable-2}),  we obtain
   \begin{align}\label{compatiable-3}\psi_l\lrcorner\omega=0.
   \end{align}
 This proves  the first relation of (\ref{assumption-2}) for $k=l$.

 By Lemma \ref{KRS cond}-(3),  we also have
 $$\delta_{h}(\psi_l\lrcorner \omega)=-\delta_{h}\psi_l\lrcorner \omega-\sqrt{-1} {\rm div}_{h}\psi_l.$$
 Thus  (\ref{compatiable-3}) implies  the second  relation of (\ref{assumption-2}) for $k=l$.  The proposition is proved.
 \end{proof}

 By  (\ref{compatiable-divergence-free}), for $v,w\in T^{0,1}M$ we have
$$\omega(v+\varphi_\tau(v),w+\varphi_\tau(w))=0.$$
 Since $T^{0,1}_{\varphi_\tau}=(id+\varphi_\tau)(T^{0,1}M)$ and $\omega$ is real, $\omega$ is still a $(1,1)$-form  with respect to the complex structure  $J_{\psi_\tau}$ defined by $\varphi_\tau$.  Namely,  $\varphi_\tau$ is compatible with the K\"ahler form $\omega$.  As a consequence,    $(\omega,  J_{\psi_\tau})$
 defines   a  family of  K\"ahler metrics  by
  \begin{align}\label{compatiable-metrics}g_\tau={\omega}(\cdot, J_{\psi_\tau}\cdot).
  \end{align}
  Hence,  we get

 \begin{cor}\label{kur}
    For any $\tau\in  B(\epsilon)$, $g_\tau$ defined  in (\ref{compatiable-metrics}) is  a family of K\"ahler metrics with the same   K\"ahler form $\omega$.
 \end{cor}

 The following proposition gives  a relationship between the first and second relations in (\ref{compatiable-divergence-free}), which means that the Kuranishi equation  is equivalent to the Cartan-Maurer equation with the divergence free gauge.

 \begin{prop}
  Suppose that $\varphi\in A^{0, 1}(M,T^{1, 0}M)$
  satisfies
       $$\overline{\partial}\varphi=\frac12[\varphi,\varphi]~{\rm with}~ {\rm div}_{h}\varphi=0.$$
 Then
       \begin{align*}\varphi\lrcorner \omega=0~{\rm and}~\delta_{h}\varphi=0.\end{align*}

 \end{prop}

 \begin{proof}
 By Lemma \ref{KRS cond}-$(5)$ and the assumption that ${\rm div}_{h}\varphi=0$,  we have
 \begin{align*}
   {\rm div}_{h}(\overline{\partial}\varphi)&=\frac12   {\rm div}_{h}[\varphi,\varphi]\\
 &=\varphi\lrcorner\partial {\rm div}_{h}\varphi\\
 &=0.
 \end{align*}
 By  Lemma \ref{KRS cond}-$(4)$  together with  the assumption ${\rm div}_{h}\varphi=0$, it follows that
 \begin{align}\label{vanish}
 \varphi\lrcorner\omega=0.
 \end{align}

 On the other hand, by Lemma \ref{KRS cond}-$(3)$ and  (\ref{vanish}),  we see that
 $$\delta_{h}\varphi\lrcorner\omega=0, $$
 which means that $\delta_{h}\varphi=0$. The proposition is proved.
 \end{proof}

 \section{Restricted entropy  $\nu(\cdot)$  and its  variations}

 From this section, we will always  assume that $(M,J_0)$ is a  Fano K\"ahler manifold which admits a  KR soliton $(\omega_{KS}, X)$. Here
 the  soliton VF  $X$  can be regarded as an element in the center of  Lie algebra $\eta_r(M, J_0)$ of a maximal reductive subgroup  ${\rm Aut}_r(M,J_0)$ of     ${\rm Aut}(M,J_0)$  \cite{TZ00}.   Namely  $(\omega_{KS}, X)$ satisfies the soliton equation,
 $$\rm {Ric}(\omega_{KS})-\omega_{KS}=\mathcal L_X \omega_{KS},$$
  where $\mathcal L_X\omega_{KS}$ is the Lie derivative  of $\omega_{KS}$ along   $X$.
   Then there is a  real smooth function $\theta=\theta_X(\omega_{KS})$ which   satisfies
   \begin{align}\label{v-potential}
     i_X\omega_{KS}=\sqrt{-1}\overline{\partial}\theta,~
     \int_Me^{\theta}\omega_{KS}^n=\int_M\omega_{KS}^n.
   \end{align}
 Thus by  Corollary \ref{kur}, there is a small ball
 $$B(\epsilon)=\{ \tau|~ \psi_\tau=\sum \tau_i e_i\in \mathcal{H}^{0, 1}_{\theta}(M,T^{1, 0}M), \sum_i \tau_i^2 < \epsilon^2\},$$
 such that the Kuranishi map:
 $$\Phi:B(\epsilon)\mapsto A^{0,1}(M,T^{1,0}M)$$
 induces
 a family of K\"ahler metrics $g_\tau$ with  the same K\"ahler form  $\omega_{KS}$.  As a consequence, any K\"ahler metric  $g$ in $2\pi c_1(M)$ with small perturbed   integral complex structure of $J_0$ can be parameterized by  the following map:
 $$L: ~B(\epsilon)\times C^{\infty}(M)\mapsto  {\rm Sym}^{2}(T^*M)$$
 with $g_{\tau,\chi}(\cdot,\cdot)=L(\psi, \chi)(\cdot, \cdot)$ satisfying
 \begin{align}\label{para-metric} g_{\tau,\chi} =(\omega_{KS}+\sqrt{-1}\partial_{J_{\psi_\tau}}\overline{\partial}_{J_{\psi_\tau}}\chi)(\cdot, J_{\psi_\tau}\cdot).
 \end{align}

 The purpose of this section is to compute the variation of  Perelman's  entropy $\lambda(\cdot)$  for K\"ahler metrics
 $g_{\tau,\chi}$.  Recall that  the Perelman's  W-functional  for   K\"ahler metrics $g$ in $2\pi c_1(M)$ is defined for a pair
  $(g, f)$  by (cf. \cite{TZ08}),
 \begin{align}\label{w-functional}W(g, f)=(2\pi)^{-n}\int_{M}[R(g)+|\nabla f|^2+f]e^{-f}\omega_{g}^n,
 \end{align}
  where   $f$  is a  real smooth function
  normalized by
 \begin{align}\label{nor}
 \int_Me^{-f}\omega_g^n=\int_M\omega_g^n.
 \end{align}
 Then $\lambda(g)$ is defined by
 $$\lambda(g)=\inf_{f}\{~W(g, f)|~\text{$(g, f)$ satisfies  (\ref{nor})}\ \}.$$
 The number $\lambda(g)$ can be attained by some $f$ (cf. \cite{Ro81}). In fact, such
 a $f$ is a solution of the equation,
 \begin{align}\label{f-function}2\triangle f+f-|D f|^2+R=\lambda(g).
 \end{align}
 In particular, $f=\theta$ if $\omega_g=\omega_{KS}$, so the minimizer of $W(g, \cdot)$  is unique near a KR soliton \cite{ZhT, SW15}.
 Thus by the relation (\ref{para-metric}),  we  get  the restricted entropy  $\nu(\cdot) $ of  $\lambda(g)$  on $\mathcal{U}_\epsilon=B(\epsilon)\times C^{\infty}(M)$ by
 $$\nu(\psi, \chi)=\lambda(L(\psi, \chi)),$$
 which is a smooth functional near $\omega_{KS}$. In fact, $\nu(\cdot) $ is analytic (cf. \cite{SW15}).

 \subsection{Variation of K\"ahler metrics}

 We  calculate  the variation of K\"ahler metrics $g_{\tau,\chi}$  at $g=g_{0,0}$ with its K\"ahler form  $\omega= \sqrt{-1}g_{i\bar j} dz^i\wedge  d\bar z^j.$
 Let  $\varphi\in A^{0, 1}(M,T^{1, 0}M)$ with  the almost complex structure  $J_\varphi$ associated with $\varphi$. Under local
 coordinates  $(z^1, \ldots z^n)$  on $(M, J_0)$, $\varphi$ is
 written as a  Beltrami differential by
 $$\varphi=\varphi_{\overline{i}}^jd\overline{z}^i\otimes\frac{\partial}{\partial z^j}.$$
  Decompose  $T_{\mathbb{C}}M$ with respect to $J_\varphi$  by
   $$T_{\mathbb{C}}M=T^{1, 0}_{J_\varphi}M\oplus T_{J_\varphi}^{0, 1}M.$$
   Then
 $$T_{J_\varphi}^{0, 1}M=(id+\varphi)T^{0, 1}M,$$
 where $\varphi$ is viewed as a map from $T^{1, 0}M\mapsto T^{1, 0}M$.

  Let
 $$
 \left\{
   \begin{array}{ll}
     J_\varphi\frac{\partial}{\partial z^i}=J^{j}_i\frac{\partial}{\partial z^j}+J^{\overline{j}}_i\frac{\partial}{\partial \overline{z}^j}\\
     J_\varphi\frac{\partial}{\partial \overline{z}^i}=J^{j}_{\overline{i}}\frac{\partial}{\partial z^j}+J^{\overline{j}}_{\overline{i}}\frac{\partial}{\partial \overline{z}^j}.\\
   \end{array}
 \right.
 $$
 Since $J_\varphi$ is real,
 $$\overline{J^{j}_i}=J^{\overline{j}}_{\overline{i}}, \overline{J^{\overline{j}}_i}=J^{j}_{\overline{i}}.$$
  Thus for a family of  $\varphi=\varphi(\tau)\in A^{0, 1}(T^{1, 0}M)$ with $\phi(0)=0$,  we get the derivative of
  almost complex structure  $J_\varphi$ as follows,
 \begin{align*}
 \frac{dJ^j_{\overline{i}}}{d\tau}|_{\tau=0}&=-2\sqrt{-1}\frac{d\varphi^j_{\overline{i}}}{d\tau}, \\
 \frac{dJ^{\overline{j}}_{\overline{i}}}{ds}|_{\tau=0}&=0.
 \end{align*}

 By Corollary \ref{kur},  $g_\tau$   given in  (\ref{compatiable-metrics}) is   a family of  K\"ahler metrics  with the fixed K\"ahler form  ${\omega}$.
  Locally, as a Riemannian tensor, $g_\tau$ is of form,
 \begin{align*}
  g_\tau=2{\rm Re}(-\sqrt{-1}{g_{k\overline{i}}}J^{k}_{\overline{j}}d\overline{z}^i\otimes d\overline{z}^j-\sqrt{-1}{g_{k\overline{i}}}J^{k}_{j}d\overline{z}^i\otimes dz^j).
 \end{align*}
 Then the  derivative of  $L$ at $(0,0)$ is given by
 $$\eta=DL_{(0,0)} (\psi, 0)=-4{\rm Re}(g_{k\overline{i}}\psi^{k}_{\overline{j}}d\overline{z}^i\otimes d\overline{z}^j)$$
 and
 $$DL_{(0,0)}(0,\chi)=2{\rm Re}(\chi_{j\overline{i}}d\overline{z}^i\otimes dz^j),$$
 where $\psi\in  \mathcal{H}_{h}^{0,1}(M,T^{1,0}M)$ and $\chi\in C^\infty(M).$
 Clearly,  $\eta$ is anti-hermitian symmetric and  $DL_{(0,0)}(0,\chi)$ is  hermitian symmetric.

 We define the divergence for a $(0,2)$-type tensor  $\eta$ by
 $${\rm div}_{h}\eta=e^{-h}{\rm div}(e^{h}\eta)=2{\rm Re}[\eta _{\overline{i}\overline{j}, k} +\eta _{\overline{i}\overline{j}}h_k)g^{k\overline{j}} d\bar z^i].$$
The following lemma shows that  the  tensor $\eta$ is also divergence-free as $\psi$.

 \begin{lem}\label{VM}
 \begin{align}\label{divergence-free-2}{\rm div}_{h}\eta=0~{\rm and}~ \eta_{\overline{i}\overline{j}, \overline{k}}=\eta_{\overline{i}\overline{k}, \overline{j}}.
 \end{align}

 \end{lem}

 \begin{proof}
   By (\ref{div}) we have
 \begin{align*}
 &({\rm div}_{h}\eta)_{\overline{i}}\\
 &=(\eta _{\overline{i}\overline{j}, k}+\eta _{\overline{i}\overline{j}}h_k)g^{k\overline{j}}\\
 &=-4(g_{l\overline{i}}(\psi^{l}_{\overline{j}, k}+\psi^{l}_{\overline{j}}h_k)g^{k\overline{j}}\\
 &=-4g_{l\overline{i}}({\rm div}_{h}\psi)^l\\
 &=0.
 \end{align*}
 The last equality comes from (\ref{compatiable-divergence-free}).  Thus   ${\rm div}_{h}\eta=0.$
 On the other hand, since $\overline{\partial}\varphi=0$, we have
 \begin{align*}
  &\eta_{\overline{i}\overline{j}, \overline{k}}\\
 &=-4g_{l\overline{i}}\psi^{l}_{\overline{j}, \overline{k}}\\
 &=-4g_{l\overline{i}}\psi^{l}_{\overline{k}, \overline{j}}\\
 &=\eta_{\overline{i}\overline{k}, \overline{j}}.
 \end{align*}
 The lemma  is  proved.
 \end{proof}

 \begin{rem}When $(M, \omega, J_0)$ is  a KE manifold, (\ref{divergence-free-2}) in Lemma \ref{VM} was verified by   Koiso \cite{Koi83}. By using the $h$-harmonic space in  (\ref{h-harmonic}), we can generalize Koiso's result for any  Fano manifold.
 The lemma  will be used in the computation of the second  variation of $\nu$ below.
 \end{rem}

 \subsection{The first variation of $\nu(\cdot)$ }
 As in \cite{TZ08},   we   have the first variation of $\lambda$ at $(M, \omega, J_0)$ as a Riemannian manifold,
 \begin{align}\label{first-variation-original-2}
 \delta\lambda= -\frac12\int_M<{\rm Ric}(g)-g+{\rm Hess}_g(f),  \delta g>e^{-f}\omega^n.
 \end{align}
 Then
  \begin{align}\label{derivative-lambda}
  \mathcal{N}(g)={\rm Ric}(g)-g+{\rm Hess}_g(f)
  \end{align}
  can be regarded as the derivative of  $\lambda$  which is a map from the space of $2$-symmetric tensors to itself.
  But for the restricted  entropy $\nu(\cdot)$, we shall define its derivative from the product  space $\mathcal U=\mathcal{H}^{0,1}_{\theta}(M,T^{1,0}M)\times C^\infty(M)$ to itself in the following.

 Recall   $f_{\psi, \chi}$  be the minimizer of $W$-functional at $g_{\tau,\chi}$ in (\ref{f-function}). We introduce a map  ${\mathbf R}: \mathcal U\to  C^\infty(M)$  by
  \begin{align}
  &{\mathbf R}(\psi, \chi)\notag\\
  &=\frac12\sqrt{-1}(\bar\partial_{J_\psi})^*_{-f_{\psi, \chi}} \cdot (\partial_{J_\psi})^*_{-f_{\psi, \chi}} \cdot ({\rm Ric}(\omega_{\tau, \chi})-\omega_{\tau, \chi}+\sqrt{-1}\partial_{J_\psi}\overline{\partial}_{J_\psi}f_{\psi, \chi})e^{-f_{\psi, \chi}-\theta}\frac{\omega_{\tau, \chi}^n}{\omega^n},\notag
 \end{align}
 where  $(\overline{\partial}_{J_{\psi}})^*_{-f_{\psi, \chi}}$ and  $(\partial_{J_{\psi}})^*_{-f_{\psi, \chi}}$ are the dual  operators  of $\overline{\partial}_{J_{\psi}}$ and  $\partial_{J_{\psi}}$,   respectively as  same as $\delta_h$ in Section 1 with respect to the following  inner product,
 \begin{align}\label{iner-product-2}(\chi_1,\chi_2)=\int_M\chi_1\chi_2e^{-f_{\psi,\chi}}\omega_{\psi, \chi}^n, ~~\chi_1,\chi_2\in L^2(M,e^{-f_{\psi,\chi}}).
 \end{align}

 Choose an unitary orthogonal  basis   $\{e_1, \ldots,  e_{l}\}$ of $\mathcal{H}_{\theta}^{0, 1}(M,T^{1, 0}M)$
  with respect to the  inner product
 $$(a, b)_{\theta}=\int_M <a, b>_{\omega_{KS}} e^{\theta}\omega_{KS}^n,~\forall~ a, b\in \mathcal{H}_{\theta}^{0, 1}(M,T^{1, 0}M).$$
 We define another map
 ${\mathbf Q}:\mathcal{U}\to \mathcal{H}_{\theta}^{0, 1}(M,T^{1, 0}M)$ by
 $${\mathbf Q}(\psi, \chi)=\sum_{i=1}^{l}(a_i(\psi, \chi)+\sqrt{-1}b_i(\psi, \chi) )e_i,$$
 where
  \begin{align*}
 a_i(\psi, \chi)&=-\frac12\int_{M}({\rm Ric}(g_{\tau,\chi})-g_{\tau,\chi}+{\rm Hess}_{g_{\tau,\chi}}(f_{\psi,\chi}), DL_{\psi, \chi}(e_i, 0))e^{-f_{\psi, \chi}}\omega^n_{\tau,\chi}
 \end{align*}
 and
 \begin{align*}
 b_i(\psi, \chi)&=-\frac12\int_{M}({\rm Ric}(g_{\tau,\chi})-g_{\tau,\chi}+{\rm Hess}_{g_{\tau,\chi}}(f_{\psi,\chi}), DL_{\psi, \chi}(\sqrt{-1}e_i, 0))e^{-f_{\psi, \chi}}\omega^n_{\tau,\chi}.
 \end{align*}

 \begin{defi}\label{first-variation-mu} We call the pair $({\mathbf Q}, {\mathbf R})$ the gradient map $\nabla\nu$ of $\nu$  on $\mathcal{U}_\epsilon$ with the inner product
 \begin{align}\label{theta-inner}(\nabla \nu, (\psi, \chi))_{\theta}=(\mathbf Q, \psi)_{\theta}+\int_M \mathbf R\chi e^{\theta}\omega_{KS}^n,
 \end{align}
 where $ (\psi, \chi)\in \mathcal U.$
 \end{defi}

 It is easy to verify that $\nabla\nu=({\mathbf Q}, {\mathbf R})$ is real analytic on $\mathcal U_\epsilon$. The following lemma gives the first variation  $\nu$ by $\nabla\nu$.

 \begin{lem}\label{gradient-pq}
   Let $(\psi_0, \chi_0)\in \mathcal{U}_\epsilon$ and $(\psi, \chi)\in \mathcal{H}_{\theta}^{0, 1}(M,T^{1, 0}M)\times C^{\infty}(M)$.
    Then
 \begin{align}\label{first-variation}\frac{d}{ds}|_{s=0}\nu((\psi_0, \chi_0)+s(\psi, \chi))=( \nabla\nu(\psi_0, \chi_0), (\psi, \chi))_{\theta}.
 \end{align}
 \end{lem}

 \begin{proof}
   For convenience, we let $g=L(\psi_0,\chi_0)$ and $f=f_{g}$ be  the minimizer of $W$ functional at $g$.   By (\ref{first-variation-original-2}),   we have
 \begin{align*}
  &\frac{d\nu((\psi_0, \chi_0)+s(\psi, \chi))}{ds}|_{s=0}\\
 &=-\frac12\int_M<  {\rm Ric}(g)-g+{\rm Hess}_g(f), DL_{(\psi_0, \chi_0)}(\psi, \chi)>e^{-f}\omega_{\psi_0, \chi_0}^n\\
 &=-\frac12\int_M< {\rm Ric}(g)-g+{\rm Hess}_g(f), DL_{(\psi_0, \chi_0)}(\psi, 0)>e^{-f}\omega_{\psi_0, \chi_0}^n\\
 &-\frac12\int_M< {\rm Ric}(g)-g+{\rm Hess}_g(f), DL_{(\psi_0, \chi_0)}(0, \chi)>e^{-f}\omega_{\psi_0, \chi_0}^n.\\
 \end{align*}
 It is easy to see that
 \begin{align*}
 I&=-\frac12\int_M<{\rm Ric}(g)-g+{\rm Hess}_g(f), DL_{(\psi_0, \chi_0)}(\psi, 0)>e^{-f}\omega_{\psi_0, \chi_0}^n\\
 &=(\mathbf{Q}(\psi_0, \chi_0), \psi)_{\theta}
 \end{align*}
 and
 \begin{align*}II
 &=-\frac12\int_M<{\rm Ric}(g)-g+{\rm Hess}_g(f), DL_{(\psi_0, \chi_0)}(0, \chi)>e^{-f}\omega_{\psi_0, \chi_0}^n\\
 &=\frac{d\lambda((\psi_0, \chi_0)+s(0, \chi))}{ds}|_{s=0}\\
 &=-\frac12\int_M<\sqrt{-1}\partial_{J_{\psi_0}}\overline{\partial}_{J_{\psi_0}}\chi , {\rm Ric}(\omega_{\psi_0, \chi_0})-\omega_{\psi_0, \chi_0}+\sqrt{-1}\partial_{J_{\psi_0}}\overline{\partial}_{J_{\psi_0}}f>
 e^{-f}\omega_{\psi_0, \chi_0}^n\\
 &=\int_M\mathbf R(\psi_0, \chi_0) \chi e^{-f}\omega_{\psi_0, \chi_0}^n.
 \end{align*}
 Thus (\ref{first-variation}) is true.

 \end{proof}

 \begin{rem}From the above, we actually have
 \begin{align}\label{first-variation-original}
 &\frac{d}{ds}|_{s=0}\nu((\psi_0, \chi_0)+s(\psi, \chi))\notag\\
 &=((\psi, \chi), \nabla\nu(\psi_0, \chi_0))_{\theta}\notag\\
 &=-\frac12\int_M<{\rm Ric}(g)-g+{\rm Hess}_g(f), DL_{(\psi_0, \chi_0)}(\psi, \chi)>e^{-f}\omega_{\psi_0, \chi_0}^n.
 \end{align}

 \end{rem}

 The  $L^2$-norm of  ${\mathbf R}$  can be controlled  by $\nabla\lambda(\cdot)=\mathcal N(\cdot)$ as follows.

 \begin{lem}\label{spectral-application-1}
 Let $k>2$ be an integer and $\epsilon$ small enough.  Then there exists a constant $C=C(\epsilon,k)$ such that
 for any $(\psi, \chi)\in \mathcal U_\epsilon$ with $\|\chi\|_{C^{4,\gamma}}\le \epsilon$  it holds
 \begin{align}\label{hessian-norm}
 \|{\mathbf R}(\psi,\chi)\|_{L^2}\le C (\int_M |{\rm Ric}(\omega_{\psi, \chi})-\omega_{\psi, \chi}+\sqrt{-1}\partial\overline{\partial}f_{\psi, \chi}|^2e^{-f_{\psi, \chi}}\omega_{\psi, \chi}^n)^{\frac{k-2}{k-1}}.
 \end{align}
 \end{lem}

 \begin{proof} Let
 \begin{align}\label{potential-R}
 {\rm Ric}(\omega_{\psi, \chi})-\omega_{\psi, \chi}+\sqrt{-1}\partial_{J_{\psi}}\overline{\partial}_{J_{\psi}}f_{\psi, \chi}=\sqrt{-1}\partial_{J_{\psi}}\overline{\partial}_{J_{\psi}}F.
 \end{align}
 We extend   ${\mathbf R}(\psi,\chi)$ to a fourth-order operator on the  Hilbert space $W_4^2(M,e^{-f_{\psi,\chi}})$   with respect  to  the  inner product  (\ref{iner-product-2})  by
 \begin{align}\label{S-operator}
 S=- (\partial_{J_{\psi}}\overline{\partial}_{J_{\psi}})^*\partial_{J_{\psi}}\overline{\partial}_{J_{\psi}}= - \overline{\partial}_{J_{\psi}}^*\cdot \partial_{J_{\psi}}^*\cdot\partial_{J_{\psi}}\overline{\partial}_{J_{\psi}}.
 \end{align}
 We  show that  $S$ is   elliptic and  self-adjoint like  the Lichnerowicz operator \cite{Cal}.

 Let  $\eta,\phi\in C^{\infty}(M)$ and $\Delta\eta=\eta_{i\bar j}g^{i\bar j}$  associated to the metric $g=L(\psi,\chi)$. Then
  in local coordinates on $(M,J_{\varphi})$,  we compute
 \begin{align*}
   (S\eta,\phi)&=\int_M\eta_{i\bar j}\overline{\phi_{k\bar l}}g^{i\bar k}g^{l\bar j}e^{-f_{\psi,\chi}}\omega_{\psi, \chi}^n\\
   &=-\int_M\eta_{i\bar j\bar k}\overline{\phi_{\bar l}}g^{i\bar k}g^{l\bar j}e^{-f_{\psi,\chi}}\omega_{\psi, \chi}^n+\int_M\eta_{i\bar j}\overline{\phi_{\bar l}}(f_{\psi,\chi})_{\bar k}g^{i\bar k}g^{l\bar j}e^{-f_{\psi,\chi}}\omega_{\psi, \chi}^n\\
   &=-\int_M\eta_{i\bar k\bar j}\overline{\phi_{\bar l}}g^{i\bar k}g^{l\bar j}e^{-f_{\psi,\chi}}\omega_{\psi, \chi}^n+\int_M\eta_{i\bar j}\overline{\phi_{\bar l}}(f_{\psi,\chi})_{\bar k}g^{i\bar k}g^{l\bar j}e^{-f_{\psi,\chi}}\omega_{\psi, \chi}^n\\
   &=-\int_M(\Delta\eta)_{\bar j}\overline{\phi_{\bar l}}g^{l\bar j}e^{-f_{\psi,\chi}}\omega_{\psi, \chi}^n+\int_M\eta_{i\bar j}\overline{\phi_{\bar l}}(f_{\psi,\chi})_{\bar k}g^{i\bar k}g^{l\bar j}e^{-f_{\psi,\chi}}\omega_{\psi, \chi}^n\\
   &=\int_M[(\Delta^2\eta)-(\Delta\eta)_{\bar j}(f_{\psi,\chi})_{l} g^{l\bar j}]\phi e^{-f_{\psi,\chi}}\omega_{\psi, \chi}^n\\
   &-\int_M[\eta_{i\bar jl}(f_{\psi,\chi})_{\bar k}g^{i\bar k}g^{l\bar j}+\eta_{i\bar j}(f_{\psi,\chi})_{l\bar k}g^{i\bar k}g^{l\bar j}-\eta_{i\bar j}(f_{\psi,\chi})_{\bar k}(f_{\psi,\chi})_{l}g^{i\bar k}g^{l\bar j}]\phi e^{-f_{\psi,\chi}}\omega_{\psi, \chi}^n\\
   &=\int_M[(\Delta^2\eta)-(\Delta\eta)_{\bar j}(f_{\psi,\chi})_{l} g^{l\bar j}]\phi e^{-f_{\psi,\chi}}\omega_{\psi, \chi}^n\\
   &-\int_M[\eta_{l\bar ji}(f_{\psi,\chi})_{\bar k}g^{i\bar k}g^{l\bar j}+\eta_{i\bar j}(f_{\psi,\chi})_{l\bar k}g^{i\bar k}g^{l\bar j}-\eta_{i\bar j}(f_{\psi,\chi})_{\bar k}(f_{\psi,\chi})_{l}g^{i\bar k}g^{l\bar j}]\phi e^{-f_{\psi,\chi}}\omega_{\psi, \chi}^n.\\
 \end{align*}
 It follows that
 \begin{align*}
   &(S\eta,\phi)\\
   &=\int_M [\Delta^2\eta-(\Delta\eta)_{\bar j}(f_{\psi,\chi})_{l}g^{l\bar j}-(\Delta\eta)_{l}(f_{\psi,\chi})_{\bar j}g^{l\bar j}\\
   &-\eta_{i\bar j}(f_{\psi,\chi})_{l\bar k}g^{i\bar k}g^{l\bar j}+\eta_{i\bar j}(f_{\psi,\chi})_{l}(f_{\psi,\chi})_{\bar k}g^{i\bar k}g^{l\bar j}] \phi e^{-f_{\psi,\chi}}\omega_g^n.
   \end{align*}
  Thus we derive
 $$S\eta=\Delta^2\eta-(\Delta\eta)_{\bar j}(f_{\psi,\chi})_{l}g^{l\bar j}-(\Delta\eta)_{l}(f_{\psi,\chi})_{\bar j}g^{l\bar j}-\eta_{i\bar j}(f_{\psi,\chi})_{l\bar k}g^{i\bar k}g^{l\bar j}+\eta_{i\bar j}(f_{\psi,\chi})_{l}(f_{\psi,\chi})_{\bar k}g^{i\bar k}g^{l\bar j}.$$
 Hence, $S$ is  an elliptic operator.   By  the regularity theorem of elliptic operator,
  $S$
  is a self-adjoint operator on  the domain $D(S)=W^2_4(M,e^{-f_{\psi,\chi}})$.

  Since $S$ is  non-negative, by the spectral theorem we see that for integer $k>0$, $x\in D(S)$ it holds
   \begin{align}\label{spec-1}
 (S^kx,x)=\int_0^{\infty}\lambda^kdE_{x,x}(\lambda),
   \end{align}
 where $E_{x,x}(\lambda)$ is the spectral decomposition defined by $S$ \cite{Rudin}. By the H\"older inequality and (\ref{spec-1}) it follows that
 \begin{align}\label{esti1}
 (S^2x,x)\leq (Sx,x)^{\frac{k-2}{k-1}}(S^kx,x)^\frac{1}{k-1}.
 \end{align}
  Thus by (\ref{potential-R}), we get
 \begin{align}\label{estimate-abstract-1}
  &\|{\mathbf R}(\psi,\chi)\|_{L^2}\notag\\
  &\leq C((\sqrt{-1}\partial_{J_{\psi}}\overline{\partial}_{J_{\psi}})^*\sqrt{-1}\partial_{J_{\psi}}\overline{\partial}_{J_{\psi}}F,
  (\sqrt{-1}\partial_{J_{\psi}}\overline{\partial}_{J_{\psi}})^*\sqrt{-1}\partial_{J_{\psi}}\overline{\partial}_{J_{\psi}}F)\notag\\
 & =C(SF,SF)\notag\\
 & =C(S^2F,F)\notag\\
 & \leq  C(SF,F)^{\frac{k-2}{k-1}}(S^kF,F)^{\frac{1}{k-1}}\notag\\
 & \le C' (\int_M |{\rm Ric}(\omega_{\psi, \chi})-\omega_{\psi, \chi}+\sqrt{-1}\partial\overline{\partial}f_{\psi, \chi}|^2e^{-f_{\psi, \chi}}\omega_{\psi, \chi}^n)^{\frac{k-2}{k-1}}(S^kF,F)^{\frac{1}{k-2}}.
 \end{align}
 Note that
 $$SF=(\sqrt{-1}\partial_{J_{\psi}}\overline{\partial}_{J_{\psi}})^*({\rm Ric}(\omega_{\psi, \chi})-\omega_{\psi, \chi}+\sqrt{-1}\partial_{J_{\psi}}\overline{\partial}_{J_{\psi}}f_{\psi, \chi}).$$
 Then
 \begin{align*}
  (S^kF,F)&=(S^{k-2}SF,SF)\\
  &=(S^{k-2}(\sqrt{-1}\partial_{J_{\psi}}\overline{\partial}_{J_{\varphi}})^*({\rm Ric}(\omega_{\psi, \chi})-\omega_{\psi, \chi}+\sqrt{-1}\partial_{J_{\psi}}\overline{\partial}_{J_{\psi}}f_{\psi, \chi}).\\
 \end{align*}
 Hence
 \begin{align}\label{Esti}
 (S^kF,F)^{\frac{1}{k-2}}<C(\epsilon)\to 0, ~{\rm as}~\epsilon\to 0.
 \end{align}
 Therefore, we get (\ref{hessian-norm}) immediately by (\ref{estimate-abstract-1}).
 \end{proof}

 \subsection{The second variation of $\nu(\cdot)$ }

 We will calculate the second valuation of $\nu$ at $(0, 0)\in\mathcal{U}_\epsilon$.
 Let $(g,f)$ be a pair  of  $W$-functional in (\ref{w-functional}) while $f$ is  a  smooth  solution of (\ref{f-function}). Denote
  $\eta=\delta g, k=\delta f$  to be the variations of $g$ and $f$ respectively. Then it  has  been shown in \cite{CHI04, TZ08} that
 \begin{align}\label{variation-Ricci}\delta({\rm Ric}(g)-g+{\rm Hess}(f))= - \Delta_{-f}(\eta)-{\rm Rm}(\eta, \cdot)-{\rm div}^*_{-f}{\rm div}_{-f}(\eta)-\frac12{\rm Hess}(v_\eta),
 \end{align}
 where  the function $v_\eta={\rm tr}\eta-2k$. Moreover,  by differentiating (\ref{f-function}), we get (cf. \cite{HM11}),
 \begin{align}\label{vari}
 \Delta_{-f}(v_\eta)+\frac{v_\eta }{2}=- \frac{1}{2} {\rm div}_{-f}{\rm div}_{-f}(\eta).
 \end{align}
 For convenience, we write the right term in (\ref{variation-Ricci}) as
 $$N(\eta)= - \Delta_{-f}(\eta)-{\rm Rm}(\eta, \cdot)-{\rm div}^*_{-f}{\rm div}_{-f}(\eta)-\frac12{\rm Hess}(v_\eta).$$
 In case that $-f=\theta$, namely $g$ is a  KR  soliton,  if
  \begin{align}\label{condition-divergence}{\rm div}_{\theta}\eta=0,
  \end{align}
  then
  $$\Delta_{\theta}( v_\eta)+ \frac{v_\eta}{2}=0.$$
 Since the first non-zero eigenvalue of $-\Delta_{\theta}$ is not less than one \cite{TZ08},   we get
  $$  v_\eta=0.$$
 Thus
 \begin{align}\label{N(h)}N(\eta)=- \Delta_{\theta}(\eta)-{\rm Rm}(\eta, \cdot),
 \end{align}
 which is  an anti-hermitian symmetric 2-tensor.

 Now we restrict  the K\"ahler metric  $g=g_{\tau,\chi}$  in $\mathcal{U}_\epsilon$.  Let
 $$(\psi, \chi), (\psi', \chi')\in \mathcal U=\mathcal{H}^{0, 1}_{\theta}(M,T^{1, 0}M)\times C^\infty(M).$$
  Then by (\ref{first-variation-original}),  we have
 \begin{align}\label{second-variation-mu-2}
   &\frac{d}{ds}|_{s=0}(\nabla\nu(s(\psi, \chi)), (\psi', \chi'))_{\theta}\notag\\
   &=-\frac12\int_M<\frac{d}{ds}_{s=0}({\rm Ric}(g_{s\psi, s\chi})+{\rm Hess}(f_{s\psi, s\chi})-g_{s\psi, s\chi}), DL_{(0, 0)}(\psi', \chi')>e^{\theta}\omega_{KS}^n\notag\\
   &=-\frac12\int_M<N(DL_{(0, 0)}(\psi,\chi), DL_{(0, 0)}(\psi',\chi')>e^{\theta}\omega_{KS}^n.
   \end{align}

 \begin{lem}\label{parts}
 $$\int_M<N(DL_{(0, 0)}(\psi, 0), DL_{(0, 0)}(0, \chi')>e^{\theta}\omega_{KS}^n=0.$$
 \end{lem}

 \begin{proof}We note that $DL_{(0, 0)}(\psi, 0)$  is anti-hermitian symmetric and (\ref{condition-divergence}) holds by Lemma \ref{VM}.  Thus  $N(DL_{(0, 0)}(\psi, 0))$   is  an anti-hermitian symmetric  2-tensor by (\ref{N(h)}). On the other hand, $DL_{(0, 0)}(0, \chi')$ is hermitian symmetric.  Hence the lemma is true.
 \end{proof}

 Analogous to the $Q$-map in Section 3.1,  we introduce a map
 $$H_1:\mathcal{H}^{0, 1}_{\theta}(M,T^{1, 0}M)\mapsto \mathcal{H}^{0, 1}_{\theta}(M,T^{1, 0}M)$$
 by
 $$H_1(\psi)=\sum_{i=1}^{l}(c_i(\psi)+\sqrt{-1} d_i(\psi))e_i,$$
 where
 $$c_i(\psi)=-\frac12\int_M<N(DL_{(0, 0)}(\psi,0), DL_{(0, 0)}(\delta_i,0)>e^{\theta}\omega_{KS}^n, i=1,2,\ldots, l,$$
 and
 $$d_i(\psi)=-\frac12\int_M<N(DL_{(0, 0)}(\psi,0), DL_{(0, 0)}(\sqrt{-1}\delta_i,0)>e^{\theta}\omega_{KS}^n, i=1,2,\ldots, l.$$

 Recall a fourth order non-positively  elliptic operator $H_2$ on $C^\infty(M)$  introduced in \cite{TZ08} by
 $$H_2=[P_0^{-1}(\overline{L}'_1L'_1)(\overline{L}_1L_1)],$$
 where
 $$P_0\chi=2\triangle_\theta\chi+\chi-(X+\overline X)(\chi),$$
 $$L_1\chi=\triangle_\theta\chi+\psi-X(\chi)$$
 and
 $$L_1'\chi =\triangle_\theta\chi-X(\chi).$$

 \begin{defi}\label{second-H}
 Define a  map $H_{\psi,\chi}: \mathcal U\to \mathcal U $ by
 $$H_{\psi,\chi}=D_{\psi,\chi}\nabla\nu$$
 and write $H=H_{0,0}$ for convenience. We call $H$ the second variation operator of $\nu(\cdot)$ at $ \omega_{KS}$.
 \end{defi}

 By the above definition and (\ref{second-variation-mu-2}),  we see that
 for any $(\psi, \chi)\in \mathcal U$, it holds
 \begin{align}\label{second-variation-H}
   \delta^2\nu_{(0,0)}((\psi, \chi), (\psi, \chi))&=\int_M <H(\psi, \chi), (\psi, \chi)>
 e^{\theta}\omega_{KS}^n\notag\\
 &=(H(\psi, \chi), (\psi, \chi))_\theta.
 \end{align}
 Moreover, by Lemma \ref{parts} we get

 \begin{prop}\label{product} $H=H_1\oplus H_2$, namely, $H(\psi,\chi)=H_1(\psi)+ H_2(\chi).$
 \end{prop}

 By a result in \cite{WY} (also see \cite{HM13}), we have  the following  explicit formula for $H_1$.

 \begin{lem}\label{wangyuanqi}
 \begin{align}\label{wy}(H_1(\psi), \psi)_{\theta}=\int_M| \psi|^2\widetilde{\theta} e^{\theta} \omega_{KS}^n,~\forall~\psi\in \mathcal{H}^{0, 1}_{\theta}(M,T^{1, 0}M),
 \end{align}
 where $\widetilde{\theta}$ differs from $\theta$ with a constant such that
 \begin{align}\label{theta}
 \int_M\widetilde{\theta} e^{\theta}\omega_{KS}^n=0.
 \end{align}
 \end{lem}

 \begin{proof}
 By Lemma \ref{VM},   $\eta=DL_{(0, 0)}(\psi, 0)$ satisfies that
 $$\eta_{i\overline{j}}=0, {\rm div}_{\theta}\eta=0~{\rm and}~\eta_{\bar i\bar j,\bar k}=\eta_{\bar i\bar k,\bar j}.$$
 For such a variation of K\"ahler metrics,   $\delta^2\lambda$ has been computed  by   the following formula  \cite{WY,HM13},
 \begin{align*}
 \delta^2\nu((\psi,0), (\psi,0))= \delta^2\lambda(\eta, \eta)=\frac12\int_M|\eta|^2(\nu(\omega_{KS})-2m+\theta)e^{\theta}\omega_{KS}^n.
 \end{align*}
 Notice that $|\eta|^2=2|\psi|^2$ and
 $$\int_M (\nu(\omega_{KS})-2m+\theta)e^{\theta}\omega_{KS}^n=0.$$
 This implies  that (\ref{wy}) holds.
 \end{proof}

 \section{A new version of $\delta^2\nu$}

 In this section,  we    give a new version of   $H_1$  and  then describe the geometry of the kernel space $Z={\rm ker} (H_1)$.  As an application,   we are able to
 prove  the local maximality of $\lambda(\cdot)$ on the space of  K\"ahler metrics associated to  the complex structures determined by  $Z$.

 \subsection{ Kernel space $Z$ of $H_1$}

  We  begin with  the following technical lemma.

 \begin{lem}\label{lie-bracket}
   Suppose that  $\varphi, \psi\in A^{0, 1}(M,T^{1, 0}M)$ satisfy
 $$\left\{
   \begin{array}{ll}
   \overline{\partial}\varphi=\overline{\partial}\psi=0\\
   \varphi\lrcorner\omega=\psi\lrcorner\omega=0.
   \end{array}
 \right. $$
 Let  $v$ be  a $(1, 0)$-VF.  Then
 $$v<\varphi, \psi>-<[v, \varphi], \psi>={\rm div}_{\theta}(v\lrcorner\overline{\psi}\lrcorner\varphi)-(v\lrcorner\overline{\varphi})\lrcorner {\rm div}_{\theta}\varphi.$$
 \end{lem}

 \begin{proof}
   Let $v=v^i\frac{\partial}{\partial z^i}, \varphi=\varphi^{l}_{\overline{k}}d\overline{z}^k\otimes \frac{\partial}{\partial z^l}, \psi=\psi^{s}_{\overline{t}}d\overline{z}^t\otimes \frac{\partial}{\partial z^s}$.
    Then
 \begin{align*}
  [v, \varphi]&=(v^i\frac{\partial \varphi_{\overline{k}}^l}{\partial z^{i}}-\varphi_{\overline{k}}^{i}\frac{\partial v^l}{\partial z^{i}})d\overline{z}^k\otimes\frac{\partial}{\partial z^l}\\
 &=(v^i\varphi_{\overline{k}, i}^l-\varphi_{\overline{k}}^{i}v^l_{, i})d\overline{z}^k\otimes\frac{\partial}{\partial z^l}\\
 \end{align*}
 and
 $$<[v, \varphi], \psi>=(v^i\varphi_{\overline{k}, i}^l-\varphi_{\overline{k}}^{i}v^l_{, i})\overline{\psi^{s}_{\overline{t}}}g^{t\overline{k}}g_{l\overline{s}}.$$
 Similarly,
 $$<\varphi, \psi>=\varphi^{l}_{\overline{k}}\overline{\psi^{s}_{\overline{t}}}g^{t\overline{k}}g_{l\overline{s}}$$
 and
 $$v<\varphi, \psi>=v^i\varphi^{l}_{\overline{k}, i}\psi^{s}_{\overline{t}}g^{t\overline{k}}g_{l\overline{s}}+v^i\varphi^{l}_{\overline{k}}\overline{\psi^{s}_{\overline{t} \overline{i}}}g^{t\overline{k}}g_{l\overline{s}}.$$
 On the other hand,
 by the condition,
  $$\varphi\lrcorner\omega=\psi\lrcorner\omega=0,$$
  we have
 $$\varphi^{l}_{\overline{k}}g^{t\overline{k}}=\varphi^{t}_{\overline{k}}g^{l\overline{k}}, \psi^{s}_{\overline{t}}g_{s\overline{l}}=\psi^{s}_{\overline{l}}g_{s\overline{t}},$$
 and by the condition,
  $$\overline{\partial}\psi=0,$$
  we have
 $$\psi^{s}_{\overline{t}, \overline{i}}=\psi^{s}_{\overline{i}, \overline{t}}.$$
 Thus
 \begin{align*}
  &v<\varphi, \psi>-<[v, \varphi], \psi>\\
 &=v^i\varphi^{l}_{\overline{k}}\overline{\psi^{s}_{\overline{t}, \overline{i}}}g^{t\overline{k}}g_{l\overline{s}}+\varphi_{\overline{k}}^{i}v^l_{, i}\overline{\psi^{s}_{\overline{t}}}g^{t\overline{k}}g_{l\overline{s}}\\
 &=v^i\varphi^{l}_{\overline{k}}\overline{\psi^{s}_{\overline{i}, \overline{t}}}g^{t\overline{k}}g_{l\overline{s}}+\varphi_{\overline{k}}^{i}v^l_{, i}\overline{\psi^{s}_{\overline{t}}}g^{t\overline{k}}g_{l\overline{s}}\\
 &=v^l\varphi^{i}_{\overline{k}}\overline{\psi^{s}_{\overline{l}, \overline{t}}}g^{t\overline{k}}g_{i\overline{s}}+\varphi_{\overline{k}}^{i}v^l_{, i}\overline{\psi^{s}_{\overline{l}}}g^{t\overline{k}}g_{t\overline{s}}\\
 &=v^l\varphi^{t}_{\overline{k}}\overline{\psi^{s}_{\overline{l}, \overline{t}}}g^{i\overline{k}}g_{i\overline{s}}+\varphi_{\overline{k}}^{i}v^l_{, i}\overline{\psi^{s}_{\overline{l}}}g^{t\overline{k}}g_{t\overline{s}}\\
 &=v^l\varphi^{t}_{\overline{k}}\overline{\psi^{k}_{\overline{l}, \overline{t}}}+\varphi_{\overline{k}}^{i}v^l_{, i}\overline{\psi^{k}_{\overline{l}}}\\
 &=v^l\varphi^{i}_{\overline{k}}\overline{\psi^{k}_{\overline{l}, \overline{i}}}+\varphi_{\overline{k}}^{i}v^l_{, i}\overline{\psi^{k}_{\overline{l}}}\\
 &=(v^l\varphi^{i}_{\overline{k}}\overline{\psi^{k}_{\overline{l}}})_{i}-v^l\varphi^{i}_{\overline{k}, i}\overline{\psi^{k}_{\overline{l}}}\\
 &=(v^l\varphi^{i}_{\overline{k}}\overline{\psi^{k}_{\overline{l}}})_{i}+(v^l\varphi^{i}_{\overline{k}}\overline{\psi^{k}_{\overline{l}}})\theta_i-v^l(\varphi^{i}_{\overline{k}, i}+\varphi^{i}_{\overline{k}}\theta_i)\overline{\psi^{k}_{\overline{l}}}\\
 &={\rm div}_{\theta}(v\lrcorner\overline{\psi}\lrcorner\varphi)-(v\lrcorner\overline{\psi})\lrcorner {\rm div}_{\theta}\varphi.
 \end{align*}
 \end{proof}

 If  $v\in \eta_r(M,J_0)$ is  a HVF on $(M, J_0, \omega_{KS})$,   there is a   potential function $P$ of $v$ such that
 \begin{align} \label{vector-potential}v=\nabla_{\overline{j}}Pg^{i\overline{j}}\frac{\partial}{\partial z^i} ={\rm grad}^{1, 0}P.
 \end{align}
  We may normalize $P$ by  the following equation  (cf. \cite{TZ00}),
 $$g^{i\overline{j}}\nabla_{\overline{j}}\nabla_{i}P+g^{i\overline{j}}\nabla_{\overline{j}}P\nabla_{i}\theta=-P.$$
 This means that
 \begin{align}\label{p-equation}{\rm div}_{\theta}(v)=v^{i}_{, i}+v(\theta)=-P.
 \end{align}

 \begin{lem}\label{lie-barcaket-2}
    Let $v={\rm grad}^{1, 0}P$ as in (\ref{vector-potential}) and $ \varphi, \psi\in \mathcal{H}^{0, 1}_{\theta}(M,T^{1, 0}M)$.
    Then
 \begin{align}\label{P-fotmula}\int_{M}<[v, \varphi], \psi>e^{\theta}\omega_{KS}^n=\int_MP<\varphi, \psi>e^{\theta}\omega_{KS}^n.
 \end{align}
 \end{lem}

 \begin{proof}
   By Corollary \ref{tech1},   we see that
 $$\left\{
   \begin{array}{ll}
   \overline{\partial}\varphi=\overline{\partial}\psi=0\\
   \varphi\lrcorner\omega_{KS}=\psi\lrcorner\omega_{KS}=0.
   \end{array}
 \right. $$
 Note that we also have
  $${\rm div}_{\theta}\varphi=0.$$
  Thus by Lemma \ref{lie-bracket} and Stoke's formula,  we get
 \begin{align*}
   \int_{M}<[v, \varphi], \psi>e^{\theta}\omega_{KS}^n&=\int_{M}v<\varphi, \psi>e^{\theta}\omega_{KS}^n.
   \end{align*}
Hence,  by (\ref{p-equation}), we prove
  \begin{align*}    \int_{M}<[v, \varphi], \psi>e^{\theta}\omega_{KS}^n&=-\int_{M}{\rm div}_{\theta}(v)<\varphi, \psi>e^{\theta}\omega_{KS}^n\\
 &=\int_MP<\varphi, \psi>e^{\theta}\omega_{KS}^n.
  \end{align*}

  \end{proof}

 \begin{rem}
   A similar relation (\ref{P-fotmula})   is obtained in \cite{CSYZ} when $\omega_{KS}$ is a  KE  metric.
 \end{rem}

Combining Lemma \ref{lie-barcaket-2} with Lemma \ref{wangyuanqi}, we have the following new version of the operator  $H_1$.

\begin{theo}\label{new-version-second-variation}
  Let $\xi={\rm Im}(X)$, where $X$ is a soliton  HVF of $(M, J_0, \omega_{KS}).$
  Then
  \begin{align}\label{H1}
   H_1(\psi)=2\sqrt{-1}\mathcal{L}_{\xi}\psi, \forall ~\psi\in \mathcal{H}^{0, 1}_{\theta}(M,T^{1, 0}M).
  \end{align}
\end{theo}

 \begin{proof}
   Note that
   $$X={\rm grad}^{1, 0}\widetilde{\theta}$$
    and
 $$g^{i\overline{j}}\nabla_{\overline{j}}\nabla_{i}\widetilde{\theta}
 +g^{i\overline{j}}\nabla_{\overline{j}}\widetilde{\theta}\nabla_{i}\widetilde\theta=-\widetilde{\theta}.$$
 Then by  Lemma \ref{lie-barcaket-2}, for any $\psi'\in \mathcal{H}^{0, 1}_{\theta}(M,T^{1, 0}M)$,  we have
 $$\int_{M}<[X, \psi], \psi'>e^{\theta}\omega_{KS}^n=\int_M\widetilde{\theta}<\psi, \psi'>e^{\theta}\omega_{KS}^n.$$
 It follows that
 \begin{align}\label{lie-derivative-integral}\int_{M}<\mathcal{L}_{X}\psi, \psi'>e^{\theta}\omega^n=\int_M\widetilde{\theta}<\psi, \psi'>e^{\theta}\omega_{KS}^n.
 \end{align}

 Let $(z^1,\ldots,z^n)$ be  holomorphic local coordinates on $(M,J_0)$. Write
  $$\psi=\psi^i_{\overline{k}}d\overline{z}^k\otimes\frac{\partial}{\partial z^i}, ~X=X^i\frac{\partial}{\partial z^i}.$$ Then
 $$\mathcal{L}_{\overline{X}}\psi=(\frac{\partial \overline{X^k}}{\partial \overline{z}^j}\psi^i_{\overline{k}}+\overline{X}(\psi^i_{\overline{j}}))d\overline{z}^j\otimes\frac{\partial}{\partial z^i}.$$
 Since
  $$\overline{\partial}\psi=0,$$
  $$\frac{\partial\psi^i_{\overline{j}}}{\partial \overline{z}^k}=\frac{\partial\psi^i_{\overline{k}}}{\partial \overline{z}^j}.$$
   Thus
 \begin{align}\label{lie-derivative-vector}
  \overline{\partial}(\overline{X}\lrcorner \psi)&=\frac{\partial \overline{X^k}\psi^i_{\overline{k}}}{\partial \overline{z}^j}d\overline{z}^j\otimes\frac{\partial}{\partial z^i}\notag\\
 &=(\frac{\partial \overline{X^k}}{\partial \overline{z}^j}\psi^i_{\overline{k}}+\overline{X^k}\frac{\partial\varphi^i_{\overline{k}}}{\partial \overline{z}^j})d\overline{z}^j\otimes\frac{\partial}{\partial z^i}\notag\\
 &=(\frac{\partial \overline{X^k}}{\partial \overline{z}^j}\psi^i_{\overline{k}}+\overline{X^k}\frac{\partial\psi^i_{\overline{j}}}{\partial \overline{z}^k})d\overline{z}^j\otimes\frac{\partial}{\partial z^i}\notag\\
 &=(\frac{\partial \overline{X^k}}{\partial \overline{z}^j}\psi^i_{\overline{k}}+\overline{X}(\psi^i_{\overline{j}}))d\overline{z}^j\otimes\frac{\partial}{\partial z^i}\notag\\
 &=\mathcal{L}_{\overline{X}}\psi.
 \end{align}
 Hence, by the fact that
 $$\delta_{\theta}\psi=0,$$
  we derive
 \begin{align}\label{vanish-part}\int_{M}<\mathcal{L}_{\overline{X}}\psi, \psi'>e^{\theta}\omega_{KS}^n=\int_{M}< \bar X\lrcorner \psi, \delta_\theta\psi'>e^{\theta}\omega^n_{KS}=0.
 \end{align}

 By (\ref{lie-derivative-integral}) and (\ref{vanish-part}), we have
 $$\int_{M}<\mathcal{L}_{\xi}\psi, \psi'>e^{\theta}\omega_{KS}^n=-\frac{\sqrt{-1}}{2}\int_M\widetilde{\theta}<\psi, \psi'>e^{\theta}\omega_{KS}^n.$$
 By Lemma \ref{wangyuanqi}, it follows that
 \begin{align}\label{h1-operator}\int_{M}<\mathcal{L}_{\xi}\psi, \psi'>e^{\theta}\omega_{KS}^n=-\frac{\sqrt{-1}}{2}\int_{M}<H_1(\psi), \psi'>e^{\theta}\omega_{KS}^n.
 \end{align}
 On the other hand,    for any $\sigma\in K_X$, where $K_X$ is the compact 1-ps of holomorphic transformations generated by $\xi$,
 $\sigma$  preserves  the K\"ahler form $\omega_{KS}$ and the function $\theta$. Then
$\sigma$ maps   $\mathcal{H}^{0, 1}_{\theta}(M,T^{1, 0}M)$ to itself.  As a consequence,
 $\mathcal{L}_{\xi}$ maps   $\mathcal{H}^{0, 1}_{\theta}(M,T^{1, 0}M)$ to itself.  Thus we prove  (\ref{H1}) by (\ref{h1-operator})
since $\psi'$ is an arbitrary  element   in  $\mathcal{H}^{0, 1}_{\theta}(M,T^{1, 0}M)$.

 \end{proof}

By Theorem \ref{new-version-second-variation}, we have
 \begin{align}\label{kernel-H}{\rm Ker}(H_1)=Z=\{\psi\in\mathcal{H}^{0, 1}_{\theta}(M,T^{1, 0}M)|~ \mathcal{L}_{\xi}\psi=0\}.
 \end{align}
Moreover,  we have the following character for the kernel $Z$.

 \begin{cor}\label{lift}
   $\psi\in Z$ if only if $\xi_{J_{\Phi(\tau\psi)}}=J_{\Phi(\tau\psi)}\xi+\sqrt{-1}\xi$ is a HVF on  $(M, J_{\Phi(\tau\psi)})$ for some  small $\tau$.
 \end{cor}

 \begin{proof}Recall that   $K_X=e^{s\xi}$  is the  1-ps   of holomorphic transformations  on $(M, J_0)$ generated by $\xi$. Then    $\xi$ is  holomorphic  on  $(M, J_{\Phi(\tau\psi)})$ if and only if
   \begin{align}
     (e^{s\xi})^* J_{\Phi(\tau\psi)}=J_{\Phi(\tau\psi)},~\forall s\in \mathbb{R},\notag
   \end{align}
  which is equivalent to
   \begin{align*}
J_{(e^{s\xi})^* \Phi(\tau\psi)}=J_{\Phi(\tau\psi)},~\forall s\in \mathbb{R}.
   \end{align*}
The latter is  also  equivalent to
   \begin{align}\label{same-solution}
(e^{s\xi})^*\Phi(\tau\psi)=\Phi(\tau\psi),~\forall s\in \mathbb{R}.
   \end{align}

   Note that $e^{s\xi}$ preserves  the K\"ahler form $\omega_{KS}$ and  $\theta$. Then
 $$(e^{s\xi})^*\psi\in \mathcal{H}^{0,1}_{\theta}(M,T^{1,0}M).$$
 Moreover, for any  $(1,1)$-form
$\varphi_k$ $(k\ge 2)$  defined  as in (\ref{psi-k}) for $\tau\psi$, we have
$$\delta_\theta( (e^{s\xi})^*\varphi_k)=0,$$
Thus by the uniqueness of  Kuranishi's solutions we get
  \begin{align}
   (e^{s\xi})^*\Phi(\tau\psi)=\Phi(\tau(e^{s\xi})^*\psi).\notag
  \end{align}
Since the   Kuranishi map is  injective,   by (\ref{same-solution}),  we derive
   \begin{align}\label{holo-4}
(e^{s\xi})^*\psi=\psi~\forall s\in \mathbb{R}.
   \end{align}
  It follows that
   $$\mathcal{L}_{\xi}\psi=0.$$
   By Theorem \ref{new-version-second-variation},  we prove that $\psi\in Z$ from (\ref{kernel-H}).
The inverse is also true. In fact, if  $\psi\in Z$, then  (\ref{same-solution}) holds for any small $\tau$.  Hence, $\psi\in Z$ if only if $ e^{s\xi}$ is  a family of holomorphisms   on  $(M, J_{\Phi(\tau\psi)})$. The corollary is proved.

 \end{proof}

\begin{rem}\label{kernel}
It has been shown  \cite{TZ08} that   the kernel of $H_2$ is   finitely  dimensional,  which is isomorphic to the linear space  generalized by the real and imaginary parts  on HVFs  of  $(M, J_0)$.  Thus   by Proposition \ref{product},
  $${\rm Ker}(H)={\rm Ker} (H_1)\oplus {\rm Ker} (H_2)$$
is also finitely  dimensional.

\end{rem}

 \subsection{Index of  $H_1$}

 Since $H_2$ is always non-positively  elliptic operator, the index (the number of positive eigenvalues) of $H$ depends only on the Lie operator $\mathcal{L}_{\xi}$ on
 $\mathcal{H}^{0,1}_{\theta}(M,T^{1,0}M)$ by Theorem \ref{new-version-second-variation}.  In the following,  we will show that it just depends  on   the cohomology group  $H^1(M,J_0,\Theta)$.

 Suppose that $e_1,\ldots,e_l$ are the eigenvectors of $\mathcal{L}_{\xi}$.  Namely there exists $\lambda_i,i=1,2,\ldots, l$ such that
 $$\mathcal{L}_{\xi}e_i=\lambda_ie_i.$$
 By (\ref{H1}) we have
 $$H_1e_i=2\sqrt{-1}\lambda_ie_i.$$
  Thus $2\sqrt{-1}\lambda_i$ is real. Denote $\lambda_i=-\sqrt{-1}\delta_i,\delta_i\in\mathbb{R}$ and let $$\theta_i=e_i,~\theta_{l+i}=\sqrt{-1}e_i,~i=1,2,\ldots,l.$$
 It follows that
 $$H_1\theta_i=2\delta_i\theta_i,~H_1\theta_{l+i}=2\delta_i\theta_{l+i}.$$
 Hence, we need to show that the number $\delta_i$ is  independent  of  choice of representation  of $[\theta_i]$.

 The following lemma can be found in the book of Kodaira \cite{Kod}.

 \begin{lem}
   If $\varphi\in A^{0,p}(M,T^{1,0}(M)), \psi\in A^{0,q}(M, T^{1,0}(M))$, we have
   \begin{align}\label{partial4}
 \overline{\partial}[\varphi,\psi]=[\overline{\partial}\varphi,\psi]+(-1)^p[\varphi,\overline{\partial}\psi].
   \end{align}
 In particular, if $Y$ is a  HVF,  and $\varphi\in A^{0,p}(M, T^{1,0}(M))$ then
 \begin{align}\label{partial3}
 \overline{\partial}\mathcal{L}_Y(\varphi)=\mathcal{L}_Y\overline{\partial}(\varphi).
 \end{align}

 \end{lem}

 \begin{lem}
 Let $\varphi\in A^{0,1}(M, T^{1,0}(M))$. Then the following is true:
 \begin{align}\label{partial}
   &\overline{\partial}\mathcal{L}_{\xi}\varphi=0, ~{\rm if}~ \overline{\partial}\varphi=0;\notag\\
 &\mathcal{L}_{\xi}\varphi=\frac{\overline{\partial}(\mathcal{L}_XY-\overline{X}\lrcorner \varphi)}{2\sqrt{-1}},
 ~{\rm if}~ \varphi=\overline{\partial}Y,Y\in T^{1,0}M.
 \end{align}

 \end{lem}

 \begin{proof}
   By (\ref{lie-derivative-vector}), we have
   \begin{align}\label{bar}
    \bar\partial\mathcal{L}_{\bar X}\varphi=0,~\forall~\varphi\in   H^1(M,J_0,\Theta).
   \end{align}
 By (\ref{partial3}),  we also have
 \begin{align}\label{partial2}
   \bar\partial\mathcal{L}_{X}\varphi=0.
 \end{align}
 Thus,
 $$\overline{\partial}\mathcal{L}_{\xi}\varphi=0.$$

 By (\ref{partial3}) we have
 $$\mathcal{L}_X\bar\partial Y=\bar\partial\mathcal{L}_XY.$$
 Together with (\ref{lie-derivative-vector}) we see that
 \begin{align*}
  \mathcal{L}_{\xi}\bar\partial Y&=\frac{\mathcal{L}_X\bar\partial Y-\mathcal{L}_{\bar X}\bar\partial Y}{2\sqrt{-1}}\\
  &=\frac{\bar\partial\mathcal{L}_X Y-\bar\partial(\bar X\lrcorner\bar\partial Y)}{2\sqrt{-1}}.
 \end{align*}
 Hence (\ref{partial}) is true.
 \end{proof}

 By (\ref{partial}), we see that
 for any $\varphi=\overline{\partial}Y,Y\in  T^{1,0}M$, it holds
 $$
  [\mathcal{L}_{\xi}\varphi]=0. $$
  Thus
 \begin{align}\label{action}
 \mathcal{L}_{\xi}[\varphi]=[\mathcal{L}_{\xi}\varphi].
 \end{align}
 This means that the eigenvalue $\theta_i$ depends  only on the operator $\mathcal{L}_{\xi}$ on
   $H^1(M,J_0,\Theta)$. Hence, we prove

 \begin{prop}\label{eigenvalue-invariant}
 The index of $H$ depends only on the operator  $\mathcal{L}_{\xi}$ on
 $H^1(M,J_0,\Theta)$.
 \end{prop}

 \subsection{Maximality of $\lambda(\cdot)$ associated to  $Z$}

 First we recall a  formula  computed  for the  $W$-functional on a  Fano manifold $(M,J)$ in \cite{TZZZ}. Let $Y$  be any HVF with ${\rm Im}(Y)$ generating  a compact 1-ps  $K_Y$ of holomorphic transformations on $(M,J)$.
 Then for any $K_Y$-invariant  K\"ahler form $\omega_g\in 2\pi c_1(M, J)$, the potential $\theta_Y(\omega_g)$ of $Y$ in (\ref{v-potential}) associated to $\omega_g$ is real. Define an invariant for $Y$ by
 $$N_{Y}(c_1(M, J))=\int_M\theta_Y(\omega_g)e^{\theta_Y(\omega_g)}\omega_g^n, $$
 which  is independent of  choice of $K_Y$-invariant  $\omega_g\in 2\pi c_1(M, J)$  \cite{TZZZ}.
 Moreover, for any  $K_Y$-invariant  $\omega_g\in 2\pi c_1(M, J)$,  we have the following formula,
 \begin{align}\label{zhu}
 W(g, -\theta_Y(\omega_g))=(2\pi)^{-n}(nV-N_{Y}(c_1(M, J))-F_{Y}^J(Y)),
 \end{align}
 where $F_{\cdot}^J(\cdot)$ is the modified Futaki-invariant on $(M,J)$ introduced in \cite{TZ02}.
 $N_{Y}(c_1(M, J))+F_{Y}^J(Y)$ is also called the $H(Y)$-invariant on $\eta_r(M, J)$  \cite{TZZZ}.

 By Corollary \ref{lift}, we know that $X =\xi_{J_{\psi_\tau}}=J_{\psi_\tau}\xi+\sqrt{-1}\xi$  is a   HVF on $(M, J_{\psi_\tau})$ for any $\tau<<1$, where $\psi_\tau=\tau\psi\in B(\epsilon)\cap Z$ and $J_{\psi_\tau}$ is associated to $\varphi_\tau=\Phi(\psi_\tau)$.  Thus by (\ref{zhu}), for any $\omega_g=\omega_{g_{\tau,\chi}}$ with $\xi(\chi)=0$  in (\ref{para-metric}), it holds
 \begin{align}\label{zhu-2}
 W(g, -\theta_X(\omega_g))=(2\pi)^{-n}(nV-N_{X}(c_1(M, J_{\psi_\tau}))-F_{X}^{J_{\psi_\tau}}(X)).
 \end{align}
 Notice that $\theta=\theta_X(\omega_{KS})$ is independent of $\psi_\tau$. Thus
 \begin{align}\label{n-invariant}
 N_{X}(c_1(M, J_{\psi_\tau}))&= \int_M\theta_X (\omega_{KS})e^{\theta_X(\omega_{KS})}\omega_{KS}^n\notag\\
 &=\int_M\theta e^{\theta}\omega_{KS}^n\notag\\
 &=N_{X}(c_1(M, J_0))
 \end{align}
 is independent of $\psi_\tau$.

 Next we show that $F_{X}^{J_{\psi_\tau}}(X)$ is also independent of $\psi_\tau$. In fact, we prove

 \begin{lem}\label{f-invariant}
 \begin{align}\label{F-invariant-2}F_{X}^{J_{\psi_\tau}}(X)=F_{X}^{J_0}(X)=0, ~\forall~\psi_\tau\in B(\epsilon)\cap Z.
 \end{align}
 \end{lem}

 \begin{proof}
 We use an argument in  \cite{Ino19} to prove the lemma. Let  $\mathcal{J}_{K_X}(M, \omega_{KS})$ be  a set  of almost $K_X$-invariant complex structures which are compatible with $\omega_{KS}$.  Then for any $J\in \mathcal{J}_{K_X}(M, \omega_{KS})$ it induces  a  Hermitian metric $g_J$.  As in   \cite{Ino19}, we define a modified Hermitian scalar curvature function on  $\mathcal{J}_{K_X}(M, \omega_{KS})$ by
 $$s_{\xi}(J)=s(J)-n+2\Box_J {\theta}-X({\theta})-{\theta}, ~\forall~J\in \mathcal{J}_{K_X}(M, \omega_{KS}),$$
 where
  $s(J)$ is the Hermitian scalar curvature  of $g_J$ (cf. \cite{Don97}) and $\Box_J$ is the Lapalace operator induced by the Chern connection  associated to $g_J$.  Now we consider a family of $J_{\psi_\tau}\in \mathcal{J}_{K_X}(M, \omega_{KS})$.
   By \cite[Proposition 3.1]{Ino19}, we have the formula,
 \begin{align}\label{Ino-formula}\frac{d}{ds}\int_M2s_{\xi}(J_{\psi_\tau}){\theta}e^{{\theta}}\omega_{KS}^n=\Omega_{\xi}(\mathcal{L}_{\xi}J_{\psi_\tau}, \dot{J}_{\psi_\tau}).
 \end{align}
 where $\Omega_{\xi}$ is a non-degenerate  2-form on $\mathcal{J}_{T}(M, \omega)$ (cf.  \cite{Don97, Ino19}).

 By Corollary \ref{lift}, we have
 $$\mathcal{L}_{\xi}J_{\psi_\tau}=0$$
 and so we get
 $$\Omega_{\xi}(\mathcal{L}_{\xi}J_{\psi_\tau}, \dot{J}_{\psi_\tau})=0.$$
 Thus by (\ref{Ino-formula}),  the quantity
   $$\int_M2s_{\xi}(J_{\psi_\tau}){\theta}e^{{\theta}}\omega_{KS}^n$$
 is independent of $\tau$.
 As a consequence,  for the normalized $\widetilde\theta$ as in (\ref{theta}),
 $$\int_M2s_{\xi}(J_{\psi_\tau}){\tilde \theta}e^{{\tilde\theta}}\omega_{KS}^n$$
 is also independent of $\tau$.
 On the other hand, it is know in \cite{TZ02} that
 $$\int_Ms_{\xi}(J_{\psi_\tau})\widetilde{\theta}e^{\widetilde{\theta}}\omega_{KS}^n=- F_X^{J_{\psi_\tau}}(X). $$
 Note that
 $$F_X^{J_0}(X)=0$$
 since $(\omega_{KS}, X)$ is the KR soliton on $(M,J_0)$. Hence,
 $$\int_Ms_{\xi}(J_{\psi_\tau})\widetilde{\theta}e^{\widetilde{\theta}}\omega_{KS}^n=0, ~\forall \tau<<1.$$
 As a consequence, we get
 $$F_X^{J_{\psi_\tau}}(X)=-\int_Ms_{\xi}(J_{\psi_\tau})\widetilde{\theta}e^{\widetilde{\theta}}\omega_{KS}^n=0.$$

  \end{proof}

 \begin{prop}\label{lml}
    Let $\psi_\tau\in B(\epsilon)\cap Z$. Then  for any $\chi\in C^{\infty}(M)$ with $\xi(\chi)=0$ it holds
    $$\lambda(g_{\tau, \chi})\leq \lambda(g_{KS}).$$
 \end{prop}

 \begin{proof}
 By  (\ref{zhu})  and (\ref{F-invariant-2}),  we have
 \begin{align*}
 \lambda(g_{\tau, \chi})&\leq W(g_{\tau, \chi}, -\theta(\omega_{g_{\tau, \chi}}))\\
 &=(2\pi)^{-n}(nV-N_{X}(c_1(M, J_{\varphi_\tau}) ).
 \end{align*}
 Thus by (\ref{n-invariant}), we get
 \begin{align*}
 \lambda(g_{\tau, \chi})&\leq(2\pi)^{-n}(nV-N_X(c_1(M, J_0)))\\
 &=\lambda(g_{KS}).
 \end{align*}

 \end{proof}

Proposition \ref{lml} can be also proved by using the equivalent formula for the modified Futaki-invariant in  \cite{WZZ} as follows.

\begin{proof}[Another proof of Proposition \ref{lml}] According to  the above  proof of Proposition \ref{lml}, we need to show that (\ref{F-invariant-2}) holds for $J_{\psi_\tau}$.  In fact, as in the proof of \cite[Lemma 2.1]{LTZ}, by the partial $C^0$-estimate, there is a large integer $L_0$ such that for any integer $k$  the family of Fano manifolds $(M, g_{\tau, 0})$ $(\tau\le\epsilon$)  can be embed into an ambient projection space $\mathbb CP^{N_k}$ by  normal  orthogonal  bases of $H^0(M, kK_M^{-L_0},g_{\tau, 0})$.  Then by introducing
 two   equivariant Riemann-Roch formulas of $S_1$ and $S_2$  with $G=(S^1)^2$-action  by
\begin{align}\label{expression}
S_1=k\frac{\partial}{\partial t}{\rm trace}(e^{\frac{sX^k+tX^k}{k}})|_{s=1,\ t=0}, \ \ S_2=\frac{1}{2}\frac{\partial}{\partial s}\frac{\partial}{\partial t}{\rm trace}(e^{\frac{sX^k+tX^k}{k}})|_{s=1,t=0},
 \end{align}
 where $X^k= (X_\alpha^k)$  is the induced HVF  as an element of  Lie algebra  $sl(N_k+1,\mathbb C)$,
$F_{X}^{J_{\psi_\tau}}(X)$ is the leading term $F_0$ (which  is a multiple of $F_1$ for a  smooth Fano variety)  in the following expansion of
\begin{align} \label{expansion}
-\frac{S_1-S_2}{kN_k}= F_0+F_1k^{-1}+o(k^{-2}).
 \end{align}
 Here
$$ {\rm trace}(e^{\frac{sX^k+tv^k}{k}})=\int_M{\rm ch}^G(kL){\rm Td}^G(M),$$
and ${\rm ch}^G(kL)$ is  the $G$-equivalent Chern character of multiple line bundle $kL=kK_M^{-L_0}$ and ${\rm Td}^G(M)$ is the  $G$-equivalent Chern character of $M$ \cite{AS}.
 Since $K_{M_\tau}^{-1}$ is the restriction of $\frac{1}{kL_0}\mathcal O(-1)$, where  $M_\tau=(M, J_\tau)\subset \mathbb CP^{N_k}$,
  $S_1$  (or $S_2$  ) as  a restriction of  derivatives of ${\rm trace}(e^{\frac{sX^k+tv^k}{k}})$ on $kL$ is same for any $\tau$ as long as $X$ can be lifted as a  HVF on $M_\tau$.
 Thus we get
 $$F_{X}^{J_{\psi_\tau}}(X)=F_{X}^{J_0}(X)=0.$$

 \end{proof}

 \section{Lojasiewicz  inequality on  the space of  K\"ahler metrics}

 In this section we prove an  inequality of  Lojasiewicz type for the functional $\nu(\cdot)$ on $\mathcal U_\epsilon$. We note  that $\nu(\cdot)$ can be defined on the $C^{k+4, \gamma}$ continuous space,
 $$W^{k+4, \gamma}(M)=  \mathcal{H}^{0, 1}_{\theta}(M,T^{1, 0}M) \times C^{k+4, \gamma}(M),$$
 where  $k\ge 0$ is   any integer.
 Clearly,  there is a natural $C^{k+4, \gamma}$-norm on $W^{k+4, \gamma}(M)$ by
 $$\|(\psi, \chi)\|_{C^{k+4, \gamma}}= \|\psi\|_{\theta} + \|\chi\|_{C^{k+4, \gamma}}.$$

 Let
 $$W_{\epsilon}^{k+4, \gamma}(M)=\{(\psi,\chi)\in W^{k+4, \gamma}(M)|~\psi\in  B(\epsilon)\}.$$
 Set an $\epsilon$-neighborhood  of  $(0,0)$ in $W_{\epsilon}^{k+4, \gamma}(M)$ by
 $$\mathcal  V_\epsilon=\{(\psi,\chi)\in W_\epsilon^{k+4, \gamma}(M)|~ \psi \in B(\epsilon), \|\chi\|_{C^{k+4, \gamma}}<\epsilon\}. $$
 We prove

 \begin{prop} \label{lojasiewsz}There are   $\epsilon>0$ and  $\alpha\in [\frac{1}{2},1)$  such that
 for any $a=(\varphi, \chi)\in\mathcal V_\epsilon $ it holds
 \begin{align}\label{loja}
 \|\nabla\nu(a)\|_{L^2}\geq c_0|\nu(a)-\nu(0)|^{\alpha},
 \end{align}
 where
 $\|\cdot\|_{L^2}$ is taken for $\nabla\nu=(Q, R)\in  W^{k, \gamma}(M)$ and $c_0>0$ is some small constant.
 \end{prop}

 \begin{proof} We will follow the argument in \cite{SW15}.
 By Remark \ref{kernel},  $W_0={\rm ker}(H)$ is  finitely  dimensional.  Then both of $W_0={\rm ker}(H)$ and
 $$W^{\perp}=\{a\in  W^{k+4, \gamma}(M)|~(a, b)_{\theta} =0,~ \forall b\in W_0\}$$
 are  closed  sets of $W^{k+4, \gamma}(M)$.
  Since  $H_2$ is elliptic,  there is a constant $C>0$ such that
 \begin{align}\label{positive-eigrn}
\parallel H(\eta)\parallel_{L^2}\geq C\parallel \eta\parallel_{L^2_4}, ~\forall~\eta\in W^{\perp}.
\end{align}

 Consider the project  map,
 $$\Phi={\rm pr}_{W^{\perp}}\nabla\nu:\mathcal{U_\epsilon}\to W^{\perp}.
 $$
 Then $\Phi$ is analysis and it satisfies that
 \begin{align}\Phi(0, 0)=0, D_{(0, 0)}\Phi={\rm pr}_{W^{\perp}}\circ H.\notag
 \end{align}
 Thus
 $$D_{(0, 0)}\Phi( \cdot): 0\oplus W^{\perp}\to  0\oplus W^{\perp}$$
  is an isomorphic.   By the implicity function theorem,  there is a neighborhood $U$  of $0$ in $W_0$  and a map $G(x): U \to W^{\perp}$
  such that
  $$\nabla\nu(x+G(x))\in W_0, ~\forall ~x\in U.$$

 Define a functional on $U$ by
 $$ F(x)=\nu(x+G(x)).$$
 Then for any  $x\in U$, $z\in W_0$, we have
  $$\frac{dG(x+tz)}{dt}|_{t=0}\in W^{\perp}.$$
  It follows that
 \begin{align*}
 \frac{dF(x+tz +G(x+tz) )}{dt}|_{t=0}&=(z+\frac{dG(x+tz)}{dt}|_{t=0}, \nabla\nu(x+G(x)))_{\theta}\\
 &=(z, \nabla\nu(x+G(x)))_{\theta}
 \end{align*}
 Thus
 $$\nabla F(x)=\nabla\nu(x+G(x))\in W_0, ~\forall x\in U.$$
 Since $F(x)$ is analytic,  by the classic Lojasiewicz inequality on $W_0$  (cf.  \cite{CS14, SW15}),  there is an $\alpha\in [\frac{1}{2},1)$ such that
 \begin{align}\label{laj-f}
 &|\nabla\nu(x+G(x))|_{L^2}\ge c_1|\nabla\nu(x+G(x))|\notag\\
 &\ge c_2|\nu(x+G(x))-\nu(0)|^{\alpha}, ~\forall ~|x|<<1.
 \end{align}

 By  Definition \ref{second-H}, we have
 $$\|H_{(\psi, \chi)}\|_{C^{k, \gamma}}\le C,~\forall~ (\psi, \chi)\in\mathcal V_\epsilon. $$
 Moreover,
 \begin{align}\label{small-h}\|H- H_{(\psi, \chi)}\|_{C^{k, \gamma}}<<1, ~{\rm as ~ long ~as}~ \|(\psi, \chi) \|_{C^{k+4, \gamma}}<<1.
 \end{align}
 For any $a=(\psi, \chi)\in \mathcal V_\epsilon$, we write
 $a=x+ G(x)+y$ 
 for some $y\in W^{\perp}$.
 Thus there exists a small  $\epsilon$ such that
 for any $a=(\psi, \chi)\in\mathcal V_\epsilon $ it holds
 \begin{align}\label{nabal-nu2}
 \nabla\nu(a)&=\nabla\nu(x+G(x)+y)\notag\\
 &=\nabla\nu(x+G(x))+\int_0^1\delta_{y}\nabla\nu(x+G(x)+sy)ds\notag\\
 &=\nabla\nu(x+G(x))+ \delta_y\nabla\nu(x+G(x)) +  \int_0^1[\delta_{y}\nabla\nu(x+G(x)+sy) - \delta_y\nabla\nu(x+G(x))] ds\notag\\
 &=\nabla F(x)+\delta_y\nabla\nu(0)+ \int_0^1 [\delta_{y}\nabla\nu(x+G(x)+sy)-  \delta_y\nabla\nu(x+G(x))]ds + o(\parallel y\parallel_{W^2_4})\notag\\
 &=\nabla F(x)+H(y)+o(\parallel y\parallel_{W^2_4}).
 \end{align}
  Since $\nabla F(x)$ is  perpendicular to  $H(y)$  with respect to $(\cdot, \cdot)_\theta$, and  $H$ is nondegenerate on $W^{\perp}$ by (\ref{positive-eigrn}), we get
 \begin{align}\label{nabla-nu} \|\nabla\nu(a)\|^2_{L^2}\geq  \|\nabla F(x)\|^2_{L^2}+c_3\| y\|^2_{W^2_4},
 \end{align}
 where $c_3>0$ is some  small constant.

 Similarly, by (\ref{small-h}), we have
 \begin{align}\label{nu-0}
  \nu(a)&=\nu(x+G(x)+y)\notag\\
  &=\nu(x+G(x))+\int_0^1<\nabla\nu(x+G(x)+sy), y>ds\notag\\
  &=\nu(x+G(x))+\frac12<\delta_y\nabla\nu(x+G(x)+y), y>\notag\\
  &+\int_0^1\int_0^1s<\delta_y\nabla\nu(x+G(x)+sty) -\delta_y\nabla\nu(x+G(x)+y), y>dtds\notag\\
  &=F(x)+\frac12<H(y), y>\notag\\
  &+\int_0^1\int_0^1s<\delta_y\nabla\nu(x+G(x)+sty)-H(y), y>dtds +o(\parallel y\parallel_{W^2_4}^2)\notag\\
   &=F(x)+\frac12<H(y), y>+ o(\parallel y\parallel_{W^2_4}^2).
 \end{align}
 Notice that
 \begin{align*}
 |<H(y), y>|&\leq \parallel y\parallel_{L^2}|\parallel H(y)\parallel_{L^2}\\
 &\leq C_1\parallel y\parallel_{L^2}\parallel y \parallel_{C^{k+4, \gamma}}\\
 &\leq C_2\parallel y\parallel_{L^2}^2\\
 &\leq C\parallel y\parallel_{W^2_4}^2.
 \end{align*}
 It follows that
 \begin{align}\label{nu-gap}|\nu(a)-\nu(0)|\leq |F(x)-F(0)|+C\parallel y\parallel_{W^2_4}^2.
 \end{align}
 Hence combining (\ref{nabla-nu}) and  (\ref{nu-gap}) together with (\ref{laj-f}), we obtain
 \begin{align*}
  \parallel \nabla\nu(a)\parallel_{L^2}^2
  &\geq c_2|F(x)-F(0)|^{2\alpha}+c_3\parallel y\parallel_{W^2_4}^2\\
  &\geq c_4(|F(x)-F(0)|+\parallel y\parallel_{W^2_4}^2)^{2\alpha}\\
  &\geq c_0|\nu(a)-\nu(0)|^{2\alpha}.
 \end{align*}
This proves (\ref{loja}).

 \end{proof}

 Recall the operator  $\mathcal{N}(\cdot)$ in (\ref{derivative-lambda}) from ${\rm Sym}^{2}(T^*M)$ to itself.
 Then we can rewrite (\ref{first-variation-original}) as
 $$(\nabla\nu(a_0), (\psi,\chi) )_{\theta}=-\frac12(\mathcal N(L(a_0)),DL_{a_0}((\psi,\chi)))_{L^2(L(a_0))}, ~\forall~a_0\in \mathcal{U}_{\epsilon}.$$
 Thus  there is  a dual operator   $(DL_{a_0})^*$ of  $DL_{a_0}$  with respect to the inner product $(\cdot,\cdot)_{\theta}$ such that
 $$((DL_{a_0})^*\mathcal{N}(L(a_0)), (\psi, \chi))_\theta
 = ( \mathcal{N}(L(a_0), DL_{a_0}((\psi,\chi)))_{L^2(L(a_0))}.$$
 As a consequence, we have
 $$\nabla\nu(a_0)=-\frac12(DL_{a_0})^*\mathcal{N}(L(a_0)).$$
 Hence,
 \begin{align}\label{norm-equi}
  (\nabla\nu(a_0),\nabla\nu(a_0))_\theta=\frac14(\mathcal N (L(a_0),DL_{a_0}(DL_{a_0})^*\mathcal{N}(L(a_0)))_{L^2(L(a_0))}.
 \end{align}

 The following  is a generalization of Lemma \ref{spectral-application-1}.

 \begin{lem}\label{esti}
 Let $k>2$ be an integer, and $\epsilon$ a small constant. Then there exists a constant $C=C(\epsilon)>0$ such that
 \begin{align}\label{hessian-norm-2}
 ((DL_{a})^*\mathcal N (L(a)),(DL_{a})^*\mathcal{N}(L(a)))_{\theta}\leq C(\mathcal N (L(a),\mathcal{N}(L(a))))^{\gamma}_{L^2(L(a))},~\forall a\in \mathcal{V}_{\epsilon}.
 \end{align}
 Here $\gamma=\frac{k-2}{k-1}<1$.
 \end{lem}

 \begin{proof}
   Let $S=(DL_{a})^*(DL_{a})$ be an operator on the Hilbert space
   $$\mathcal{H}=\mathcal{H}^{0,1}_{\theta}(M,T^{1,0}M)\times L^2(M), $$
    whose domain $D(S)$ contains $ \mathcal{H}^{0,1}_{\theta}(M,T^{1,0}M)\times C^\infty(M)$. As  same as  the operator in  (\ref{S-operator}),  $S$ is a self-adjoint fourth-order non-negative  elliptic operator. Thus by the spectral theorem,  for integer $k>0$, $x\in D(S)$, it holds
   \begin{align}\label{spec}
 (S^kx,x)_{\theta}=\int_0^{\infty}\lambda^kdE_{x,x}(\lambda).
   \end{align}
  By the H\"older inequality and (\ref{spec}) we get
 \begin{align}\label{estimate}
 (S^2x,x)_{\theta}\leq (Sx,x)_{\theta}^{\frac{k-2}{k-1}}(S^kx,x)_{\theta}^\frac{1}{k-1}.
 \end{align}

 For  any $a\in \mathcal{V}_{\epsilon}$, we  decompose  $\mathcal N (L(a))$ into
 \begin{align}\label{decom}\mathcal N (L(a))=DL_ax+y,
  \end{align}
 where $x\in \mathcal{H}^{0,1}(M,T^{1,0}M)\times C^{\infty}(M)$ and $y\in {\rm Im}(DL_a)^{\perp}$. Since
 $$(z,(DL_a)^*(y))_{\theta}=(DL_a(z),y)_{L^2(L(a))}=0, ~ \forall~ z\in \mathcal{H}^{0,1}(M,T^{1,0}M)\times C^{\infty}(M), $$
  $$(DL_a)^*(y)=0.$$
   Thus by (\ref{estimate}) and the fact that $S$ is self-adjoint, we have
 \begin{align}\label{estimate-abstract}
  &((DL_{a})^*\mathcal N (L(a)),(DL_{a})^*\mathcal{N}(L(a)))_{\theta}\notag\\
  &=((DL_{a})^*DL_{a}x,(DL_{a})^*DL_{a}x)_{\theta}\notag\\
  &=(Sx,Sx)_{\theta}\notag\\
  &=(S^2x,x)_{\theta}\notag\\
  &\leq (Sx,x)_{\theta}^{\gamma}(S^kx,x)_{\theta}^{1-\gamma}\notag\\
  &=(DL_a x,DL_a x)_{L^2(L(a))}^\gamma(S^kx,x)_{\theta}^{1-\gamma}\notag\\
  &\leq\|\mathcal N (L(a)))\|_{L^2(L(a))}(S^kx,x)_{\theta}^{1-\gamma}.
 \end{align}

 On the other hand,
 \begin{align*}
  (S^kx,x)_{\theta}&=(S^{k-1}x,Sx)_{\theta}\\
  &=(S^{k-2}(DL_a)^*DL_ax,(DL_a)^*DL_ax)_{\theta}\\
  &=(S^{k-2}(DL_a)^*\mathcal N (L(a)),(DL_a)^*\mathcal N (L(a)))_{\theta}.
 \end{align*}
 Then
 \begin{align}\label{Sk}
   (S^kx,x)_{\theta}^{1-\gamma}\leq C,
 \end{align}
 as long as $a\in \mathcal{V}_{\epsilon}$.
 Hence, combining  (\ref{estimate-abstract}) and (\ref{Sk}),  we   derive (\ref{hessian-norm-2}).
 \end{proof}

 \begin{rem}
   In order to apply the spectral theorem,  the fourth-order  operator $S$  should be  self-adjoint as in (\ref{S-operator}). In the other words,  we shall define  the domain $D(S)$ of S.   By  decomposing  $S$ into
   $$S(\varphi,\chi)=(S_{11}\varphi+S_{12}\chi,S_{21}\varphi+S_{22}\chi),$$
    it suffices to consider the domains  $D(S_{12})$ and $D(S_{22})$.  On the other hand,  by  (\ref{S-operator}),
 $$S_{22}=(\sqrt{-1}\partial_{J_{\varphi}}\bar\partial_{J_{\varphi}})^*(\sqrt{-1}\partial_{J_{\varphi}}\bar\partial_{J_{\varphi}})$$
    is a fourth-order elliptic self-adjoint  operator.  Thus  $D(S_{22})=W^2_4(M)$.   Note that $S_{12}$ is a  second order operator.  Hence,
     $$ D(S)=\mathcal{H}^{0,1}_{\theta}(M,T^{1,0}M)\times W^2_4(M)$$
     and so  $S$  is   self-adjoint  on  $D(S)$.
 \end{rem}

 \begin{cor}\label{loj-inequ-ZhZ}
 There are  constant $c_0>0$ and  number $\alpha'\in(\frac12,1)$ such that
   \begin{align}\label{lojas}
 \|\mathcal{N}(L(a))\|_{L^2(L(a))}\geq c_0\|\lambda(L(a))-\lambda(L(0))\|^{\alpha'}, \forall a\in \mathcal{V}_{\varepsilon}.
 \end{align}
 \end{cor}

 \begin{proof}
   By (\ref {norm-equi}), Lemma \ref{esti} and  Proposition \ref{lojasiewsz}, we have
   \begin{align*}
    \|\mathcal{N}(L(a))\|_{L^2(L(a))}&\geq C\|\nabla\nu(a)\|_{L^2}^{\frac{1}{\gamma}}\\
    &\geq C\|\nu(a)-\nu(0)\|^\frac{\alpha}{\gamma}\\
    &=C\|\lambda(L(a))-\lambda(L(0))\|^\frac{\alpha}{\gamma}.
   \end{align*}
   Set $\alpha'=\frac{\alpha}{\gamma}$.  Then $\alpha'\in (\frac12,1)$  when $k>>1$. Hence, (\ref{lojas}) is true.
 \end{proof}

 \section{ Proof of Theorem \ref{main-theorem-3}}

In this section, we prove  Theorem \ref{main-theorem-3}.  Let  $\widetilde{f}(t)$ be the minimizer of $W$-functional as in (\ref{f-function}) for the solution $\widetilde{g}(t)$ of KR flow (\ref{kr-flow}). Let $X(t)$ be a family of  gradient  vector fields defined by
 $$X(t)=\frac12{\rm grad}_{\widetilde{g}(t)}\widetilde{f}(t),$$
which  generates a family of differential transformations $F(t)\subseteq {\rm Diff}(M)$.
 Then $g(t)=F(t)^*\widetilde{g}(t)$  is a solution of the following modified Ricci flow,
  \begin{align}\label{normalization-g}
       \frac{\partial g(t)}{\partial t}=-{\rm Ric}(g(t))+g(t)+{\rm Hess}_t(f(t)),
 \end{align}
 where $f(t)=F(t)^*\widetilde{f}(t)$ is  the minimizer of $W$-functional for ${g}(t)$. It follows that
\begin{align}\label{monotonicity}
 \|\dot{g}(t)\|_{L^{2}(t)}=\|\mathcal{N}(t)\|_{L^2(t)}.
\end{align}
Here
 $\mathcal N(t)=  {\rm Ric}(g(t))-g(t)+{\rm Hess}_t(f(t))$
 just as  one in (\ref{derivative-lambda}).

 Let   $\omega(t)=F(t)^*\widetilde{\omega}(t)$.  Then $\omega(t)$ is just the K\"ahler form of $g(t)$.  Thus by (\ref{normalization-g}), we get
 \begin{align}
   \frac{d \omega(t)}{dt}= -{\rm Ric}(\omega(t))+\omega(t)+\sqrt{-1}\partial_{J(t)}\bar\partial_{J(t)}(f(t))\notag
   \end{align}
   and
   \begin{align}\label{normalization-omega}
   \omega(t)(\cdot,\frac{dJ(t)}{dt})=D_tf(t),
 \end{align}
 where $D_tf$ denotes  the anti-Hermitian part of ${\rm Hess}(f)$.  Hence,
 \begin{align}\label{contral-complex-structure}
   \|\dot{J}(t)\|_{L^2(t)}&=\|D_tf(t)\|_{L^2(t)}\notag\\
   &\leq \|\mathcal{N}(t)\|_{L^2(t)}\notag\\
   &=\|\dot{g}(t)\|_{L^{2}(t)}.
 \end{align}

 We  first prove  the following convergence theorem.

 \begin{theo}\label{main-theorem-2}
   There exists  a small  $\epsilon$ such that
 for any $(\tau, \chi)\in\mathcal V_\epsilon $ with $\psi_\tau=\sum\tau_i e_i\in Z$ and $\xi(\chi)=0$ the flow (\ref{normalization-g}) with the initial metric $g=L(\tau,\chi)$
  converges smoothly  to a  KR soliton $(M, J_\infty, g_{\infty})$.
    Moreover, the convergence is fast  in the polynomial rate.
 \end{theo}

 \begin{proof}We  need to prove that the normalized flow $g(t)$ of (\ref{normalization-g}) is uniformly bounded in $C^{k+2,\gamma}$-norm.
  Fix a small number $\delta_0$ we consider
  \begin{align}\label{T-set}T=T_{\tau, \chi}=\sup\{t|~ \|g(t)-g_{KS}\|_{C^{k+2,\gamma}}<\delta_0, ~ \|J(t)-J_{0}\|_{C^{k+2,\gamma}}<\delta_0\}.
  \end{align}
  Then it suffices to show that $T=\infty$. By the stability of Ricci flow for the short time, $T_{\tau, \chi}\ge T_0>0.$ In fact, for the KR  flow (\ref{normalization-g}),   $T_0$ can be made any large as long as $\tau$ and $\chi$ are small enough.

   By the Kuranishi   theorem  for the completeness of  deformation space of complex structures  \cite{MK}, for any $t<T$,  there is a
 $K(t)\in {\rm Diff(M)}$ closed to the identity such that
 \begin{align}\label{k-difformation}K(t)^*J(t)=J_{\varphi_t}, ~\varphi_t=\Phi(t)=~\Phi(\psi_t)
 \end{align}
 for some $\psi_t\in \mathcal{H}^{0,1}_{\theta}(M,T^{1,0}M)$ with $\|\psi_t\|<\epsilon_0$, where $\epsilon_0=\epsilon$ as chosen in  (\ref{epsilon-choice}) and (\ref{lojas}) in Corollary \ref{loj-inequ-ZhZ}. In fact,  by the construction of $K(t)$ in \cite[Theorem 3.1]{MK}, we have the estimate
 \begin{align}\label{chi-small-improve}&\|\psi_t\|_{L^2(M)}=O(\|\varphi_t\|_{L^2(M)})\le \epsilon_0,\notag\\
 &K(t)^*\omega(t)=(1+O(\|\varphi_t\|_{L^2(M)}))\omega(t).
 \end{align}
 Thus there is a  smooth function $\chi_t$ such that
 \begin{align}\label{potential}
 \omega_{K(t)^*g(t)}=K(t)^*\omega(t)=\omega_{KS}+\sqrt{-1}\partial_{J_{\varphi_t}}\bar\partial_{J_{\varphi_t}}\chi_t.
 \end{align}
 Here $\chi_t$ can be  normalized by
 $$\int_M \chi_t\omega_{KS}^n=0.$$
 It follows that
 \begin{align}\label{metric-comparison}
 \|\Delta_{g_{KS}}\chi_t\|_{C^{k+2,\gamma}} & \le 2 \|\Delta_{g_{\psi_t,0}}\chi_t\|_{C^{k+2,\gamma}}\notag\\
 &=2\|{\rm trace}_{\omega_{KS}}(K(t)^*\omega(t))-n  \|_{C^{k+2,\gamma}}\notag\\
 & \le 2\|(1+\delta_1) {\rm trace}_{g_{KS}}(g(t))-n  \|_{C^{k+2,\gamma}} \le \delta_2(\delta_0)<<1.
 \end{align}
 Hence, we get
 \begin{align}\label{chi-small} \|\chi_t\|_{C^{k+4,\gamma}(M)}\le \epsilon_0
 \end{align}
 as long as $\delta_2$ is chosen small enough.
 Note that   $\|\mathcal{N}(g)\|_{L^2(g)}$ is invariant under the action of ${\rm Diff}(M)$.   Therefore, by Corollary \ref{loj-inequ-ZhZ},
  we obtain the following Lojasiewicz  inequality,
 \begin{align}\label{lajw-inequ-3}\|\mathcal{N}(g(t))\|_{L^2(g(t))} =\|\mathcal{N}(L(\psi_t, \chi_t))\|_{L^2(L(\psi_t,\chi_t))}\geq c_0\|\lambda(g(t))-\lambda(g_{KS})\|^{\alpha},
 \end{align}
 where $\alpha\in (\frac{1}{2}, 1)$.

 Note that $\widetilde{g}(t)$ are all $K_X$-invariant for any $t\ge 0$.   Then by Proposition \ref{lml}, we have
 \begin{align}\lambda(\widetilde{g}(t))\leq \lambda(g_{KS}), ~t>0.
 \notag
 \end{align}
 and so,
 \begin{align}\label{lower-bound-2}\lambda(g(t))\leq \lambda(g_{KS}), ~t>0.
 \end{align}
  Thus by (\ref{normalization-g}),  for any $\beta>2-\frac{1}{\alpha}$, we get
 \begin{align}\label{lambda-decay-1}
  &\frac{d}{dt}[\lambda(g_{KS})-\lambda(g(t))]^{1-(2-\beta)\alpha}\notag\\
 &=-(1-(2-\beta)\alpha)[\lambda(g_{KS})-\lambda(g(t))]^{-(2-\beta)\alpha} \frac{d}{dt}\lambda(g(t))\notag\\
 &=-\frac{(1-(2-\beta)\alpha)}{2}[\lambda(g_{KS})-\lambda(g(t))]^{-(2-\beta)\alpha}\int_M<\mathcal N(t),\mathcal N(t)>e^{-f_t}\omega_t^n.
 \end{align}
 By  (\ref{lajw-inequ-3}),  it follows that
 \begin{align}
 \frac{d}{dt}[\lambda(g_{KS})-\lambda(g(t))]^{1-(2-\beta)\alpha}&\leq -C(\int_M<\mathcal N(t),\mathcal N(t)>e^{-f_t}\omega_t^n)^{\frac{\beta}{2}}\notag\\
 &=-C \parallel \dot{g}(t)\parallel^{\beta}_{L^2(t)}.\notag
  \end{align}
 As a consequence,
 \begin{align}\label{decay-2}
   \int_{t_1}^{t_2}\parallel \dot{g}(t)\parallel^{\beta}_{L^2(t)}dt&\leq C(\beta)[\lambda(g_{KS})-\lambda(g(t_1))]^{1-(2-\beta)\alpha}
   \end{align}
   Hence,  for any  $(\tau, \chi)\in \mathcal V_\epsilon$, we prove
 $$\int_1^T\| \dot{g}(t)\|_{L^2(t)}^{\beta}dt
 \leq C(\beta)[\lambda(g_{KS})-\lambda(g(0))]^{1-(2-\beta)\alpha}
\le \delta_3(\epsilon).$$
 Similarly,  it also holds
 $$\int_1^T\| \dot{J}(t)\|_{L^2(t)}^{\beta}dt<\delta_4(\epsilon).$$

 By (\ref{lambda-decay-1}) and (\ref{lajw-inequ-3}),  we see that
 $$\frac{d}{dt}[\lambda(g_{KS})-\lambda(g(t))]^{1-(2-\beta)\alpha}\leq -C'(\lambda(g_{KS}) -\lambda(g(t)))^{\alpha\beta}.$$
 Then
 $$\frac{d}{dt}[\lambda(g_{KS})-\lambda(g(t))]^{1-2\alpha}\geq C>0.$$
 It follows that
 \begin{align}\label{lambda-decay-4}\lambda(g_{KS})-\lambda(g(t))\leq C''(t+1)^{-\frac{1}{2\alpha-1}}.
 \end{align}
 On the other hand,  by  the interpolation inequalities for tensors,   for $\beta\in (2-\frac{1}{\alpha}, 1)$ and  any integer $p\geq 1$, there exists $N(p)$ which is independent of  $t$, such that
 \begin{align*}
  \| \dot{g}(t)\|_{L^2_{p}(t)}&\leq C(p)\| \dot{g}(t)\|^{\beta}_{L^2(t)}\| |\mathcal N(t)|^{1-\beta}_{L^2_{N(p)}(t)}\\
 &\leq C(p)\parallel \dot{g}(t)\parallel^{\beta}_{L^2(t)}, ~\forall ~t<T.
 \end{align*}
 Similarly,  we have
 \begin{align}\label{complex-structure}
 \| \dot{J}(t)\|_{L^2_{p}(t)} \leq C(p)\parallel \dot{J}(t)\parallel^{\beta}_{L^2(t)}, ~\forall ~t<T.
 \end{align}
  Hence, by (\ref{decay-2}) and (\ref{lambda-decay-4}),  we derive
 \begin{align}\label{decay-g}
   \parallel g(t_1)-g(t_2)\parallel_{C^{k+2, \gamma}}&\leq \int_{t_1}^{t_2}\parallel \dot{g}(t)\parallel_{C^{k+2, \gamma}}\notag\\
 &\leq C\int_{t_1}^{t_2}\parallel \dot{g}(t)\parallel_{L^2_p(t)}\notag\\
 &\leq C\int_{t_1}^{t_2}\parallel \dot{g}(t)\parallel^{\beta}_{L^2(t)}\notag\\
 &\leq C(\beta)[\lambda(g_{KS})-\lambda(g(t_1))]^{1-(2-\beta)\alpha}\notag\\
 &\leq C(\beta)(t_1+1)^{-\frac{1-(2-\beta)\alpha}{2\alpha-1}}, ~ \forall~ t_2>t_1.
 \end{align}
Therefore, we can choose a large $T_0$ such that
  $$\|g(t)\|_{C^{k+2, \gamma}}<\delta_0, ~\forall~ t\ge T_0.$$
 Similarly we can use  (\ref{contral-complex-structure})  and (\ref{complex-structure})  to prove that
 \begin{align}\label{small-j}\|J(t)\|_{C^{k+2, \gamma}}<\delta_0.~\forall ~t\ge T_0.
 \end{align}
 As a consequence,  we prove  that $T=\infty$. Furthermore, we can show
   that the limit of $g(t)$ is a   KR soliton  by (\ref{monotonicity}) (cf. \cite{TZ07, ZhT}).
  The convergence speed comes from  (\ref{decay-g}).

 \end{proof}

\begin{rem}\label{DS-result}By (\ref{lambda-decay-4}) and  a result of  Dervan-Sz\'ekelyhidi  \cite{DeS} (also see \cite{WZ1}), we know that
 \begin{align}\label{maximum}\sup_{\omega_g\in 2\pi c_1(M, J_{\psi_\tau})}\lambda(g)=\lim_t\lambda(g(t))= \lambda(g_{KS}).
 \end{align}
 In particular, (\ref{lower-bound-2}) holds for  KR flow (\ref{kr-flow}) with any initial metric in $2\pi c_1(M,  J_{\psi_\tau})$. Thus
  the condition $\xi(\chi)=0$ can be removed  in Theorem \ref{main-theorem-2} according to the above proof \footnote{The convergence part also comes from the Hamilton-Tian conjecture and the uniqueness result in   \cite{HL, WZ1}. }.
  \end{rem}

  We also remark  that the limit complex structure $J_\infty$ in  Theorem \ref{main-theorem-2} may be different with the original one $J_{\psi_\tau}$. But  the soliton VF of $(M_\infty, J_\infty)$ must be  conjugate to $X$. In fact, we have the following analogy  of \cite[Proposition 5.10]{WZ20} for the uniqueness of soliton  VFs .

  \begin{lem}\label{vector-gap} There is a small $\delta_0$ such that for any KR soliton  $( M, J', g_{KS}' )$ with
   \begin{align}\label{small-GH}{\rm dist}_{CG, C^3}( ( M', g_{KS}'),    ( M, g_{KS}))\le \delta_0,
   \end{align}
  $X$ can be lifted to a  HVF on  $(M, J')$ so that it is a soliton    VF of  $(M, J', g_{KS}')$.
Moreover,
\begin{align}\label{identity-5}\lambda(g_{KS}')=\lambda(g_{KS}),
 \end{align}

   \end{lem}

  \begin{proof}As in the second  proof of Proposition \ref{lml}, by the partial $C^0$-estimate, one can embed any  KR soliton  $( M, J'; \omega_{KS}', X')$ with  satisfying (\ref{small-GH})  into  an ambient projection space $\mathbb CP^{N}$.  Then the lemma turns to prove that
  there is some  $\sigma\in {\rm U}(N + 1,\mathbb C)$ such that
   $$ X'= \sigma\cdot  X \cdot \sigma^{-1}.$$
 Thus one can use the argument in  \cite[Proposition 5.10]{WZ20} to get   the  compactness of  soliton VFs   in ${\rm sl}(N+1, \mathbb C)$ and  prove the lemma by the uniqueness result of soliton   VFs  \cite[Proposition 2.2]{WZ20}.

On the other hand, by (\ref{small-GH}), there is a $F\in {\rm Diff(M)}$ such that
 $$\|F^*J'-J_0\|_{C^{3}}\le \delta_0'.$$
 Then  by the Kuranishi's theorem,   there are  $K\in {\rm Diff(M)}$ and $\tau'\in B(\epsilon)$ such that
 $$ K^*(F^*J')=J_{\psi_{\tau'}}.$$
Thus, we may assume that $( M,  \omega_{KS}' )$  is a KR soliton with respect to $J'=J_{\psi_{\tau'}}$.
 Since both of $\theta_X(\omega_{KS}')$ and  $\theta_X(\omega_{KS})$ are minimizers of the $W$-functional respect to $\omega_{KS}'$ and  $\omega_{KS}$,   by   (\ref{zhu-2})-(\ref{F-invariant-2}) in Section 4,  we obtain
 $$ \lambda(g_{KS}')=W(g_{KS}', -\theta_X(\omega_{KS}'))= W(g_{KS}, -\theta_X(\omega_{KS}))= \lambda(g_{KS}).
 $$

\end{proof}

  By Lemma \ref{vector-gap},  we can finish the proof of Theorem \ref{main-theorem-3}.

\begin{proof}[Proof of Theorem \ref{main-theorem-3}]  By the uniqueness of limits of  KR flow   \cite{HL, WZ1}, we may assume that $\widetilde\omega_0=g=L(\tau,\chi)$ with  $(\tau, \chi)\in\mathcal U_\epsilon $ and  $\xi(\chi)=0$.
Thus the convergence part of theorem comes from Theorem \ref{main-theorem-2}.
  Moreover, by Lemma \ref{vector-gap},
 $X$ can be lifted to a  HVF on  $(M, J_\infty)$ so that it is a soliton    VF of  $(M, J_\infty, g_{\infty})$ and (\ref{identity}) is satisfied.
(\ref{identity}) also comes from  (\ref{lambda-decay-4})  as well as the convergence speed of $g(t)$ with the polynomial rate comes from  (\ref{decay-g}).

\end{proof}

 \section{Applications of  Theorem \ref{main-theorem-3}}

In this section, we first prove Theorem \ref{existenceThm}.
Recall that   a  {\it special degeneration } on a  Fano manifold $M$ is a  normal variety $\mathcal M$ with a $\mathbb C^*$-action which consists of three ingredients \cite{Ti97}:
\begin{enumerate}
\item an flat $\mathbb C^*$-equivarant map $\pi: \mathcal M\to \mathbb C$ such that $\pi^{-1}(t)$ is biholomorphic to $M$ for any $t\neq 0$;

\item an holomorphic line bundle $\mathcal L$ on $\mathcal M$ such that $\mathcal L|_{\pi^{-1}(t)}$ is isomorphic to $K_M^{-r}$ for some integer $r>0$  and any $t\neq 0$;

    \item a center $M_0= \pi^{-1}(t)$ which is a $Q$-Fano variety.
\end{enumerate}

The following definition can be found in \cite{WZZ} (also see \cite{Xi, BN, DS16}, etc.).

\begin{defi}A Fano manifold  $M$ is  modified $K$-semistable with respect to a HVF $X$ in $\eta(M)$ if the modified  Futaki-invariant $F(\cdot)\ge 0$  for any  {\it special degeneration } $\mathcal M$   associated to  a 1-ps $\sigma_t$  induced by  a $\mathbb C^*$- action, which  communicates  with the one-parameter subgroup  $\sigma_t^X$ associated to the lifting of $X$ on  $\mathcal M$.
In addition that  $F(\cdot)= 0$
if only if  $\mathcal M\cong M\times  \mathbb C$, $M$ is  called  modified $K$-polystable with respect to $X$.

\end{defi}

For a Fano manifold $(M, J_{\psi_\tau})$ in the deformation space in Theorem \ref{main-theorem-3}, we introduce

 \begin{defi}\label{GIT-notion} A Fano manifold $(M, J_{\psi_\tau})$ is called modified    $K$-semistable (modified $K$-polystable) in the deformation  space  of complex structures on  $(M,J_0)$ which admits  a KR soliton  $(g_{KS}, X)$ if $X$ can be lifted to a  HVF on $(M, J_{\psi_\tau})$ and $(M, J_{\psi_\tau})$  is  modified $K$-semistable (modified $K$-polystable) with respect to $X$.

\end{defi}

  \begin{proof}[Proof of Theorem \ref{existenceThm}]  By Lemma \ref{vector-gap}, the soliton VF on  $(M, J_{\psi_{\tau}})$ is conjugate to $X$.  Thus we need to prove  the sufficient part since the  modified $K$-polystability is a  necessary condition for the existence of KR solitons (cf. \cite{BN}).

  By Theorem \ref{main-theorem-3} and Lemma \ref{vector-gap},
  the  KR  flow  $g(t)$ converges smoothly  to a  KR soliton $(M, J_\infty, g_{\infty})$ with respect to $X$  for  the initial K\"ahler metric $g_{\tau,0}$ on $(M, J_{\psi_{\tau}})$. Then as in the proof of \cite[Lemma 2.1]{LTZ}, by the partial $C^0$-estimate, there is a large integer $L_0$ such that   the family of K\"ahler  manifolds $(M, g(t))$  and the limit manifold $(M, J_\infty, g_{\infty})$ can be embed into an ambient projection space $\mathbb CP^{N}$ by  normal  orthogonal  bases of $H^0(M, K_M^{-L_0},g(t))$ and the images $\tilde M_t$ of $(M, g(t))$ converges smoothly to the image $\tilde M_\infty$ of $(M, J_\infty, g_{\infty})$ in $\mathbb CP^{N}$.    Thus  there is  a family of group $\sigma_t\in {\rm SL}(N+1,\mathbb C)$ such that
  $$\sigma_t (\tilde M_{t_0})=\tilde M_t,~\forall t\ge t_0.$$
  Moreover, we may assume that
  $$\sigma_t\cdot \sigma_s^X= \sigma_s^X\cdot \sigma_t,$$
  by modifying the  base of $H^0(M, K_M^{-L_0},g(t))$ after a  transformation in ${\rm SL}(N+1, \mathbb C)$  (cf. \cite{DS16}).
  Here $X$ can be  lifted to a  HVF in $\mathbb CP^N$ such that it is tangent to each $\tilde M_t$.

  Set a reductive subgroup of ${\rm SL}(N+1,\mathbb C)$ by
  $$G=\{\sigma\in {\rm SL}(N+1,\mathbb C)|~\sigma\cdot \sigma_s^X= \sigma_s^X\cdot \sigma\}.$$
  Then
  $$G_c=\{\sigma\in G|~\sigma(\tilde M_\infty)=\tilde M_\infty\}\subset {\rm Aut}_r(\tilde M_\infty),$$
  where ${\rm Aut}_r(\tilde M_\infty)$ is a reductive subgroup of ${\rm Aut}(\tilde M_\infty)$ which contains $\sigma_s^X$. Note that $M_\infty$ admits a KR soliton and $X$  is a center of $\eta_r(\tilde M_\infty)$ by \cite{TZ02}. Thus
  $G_c={\rm Aut}_r(\tilde M_\infty)$ is a reductive subgroup of $G$. By  GIT, there is  a  $\mathbb C^*$-action in $G$  which induces
a  {\it smooth  degeneration } $\mathcal M$ on the Fano manifold   $\tilde M_{t_0}$   with the center $\tilde M_\infty$. Clearly, this $\mathbb C^*$-action  communicates  with   $\sigma_s^X$.  It follows that the corresponding modified  Futaki-invariant $F(\cdot)= 0$
by the fact  $\tilde M_\infty$ admitting  a KR soliton. Hence $\mathcal M$ must be  a trivial  degeneration and we get $M\cong M_\infty$.
This proves that $M$ admits a KR soliton.

 \end{proof}

Let $T$ be a torus subgroup of ${\rm Aut}_r(M, J_0)$ which contains the soliton VF $X$. Then  by Proposition \ref{eigenvalue-invariant} and Corollary \ref{lift}, we see that any   $T$-equivalent  subspace $H_{X}^1(M,J_0,\Theta)\subset H^1(M,J_0,\Theta) $ is included in  $Z$, which is invariant under ${\rm Aut}_r(M, J_0)$. Thus by Theorem \ref{existenceThm} we actually prove
 the following  existence result   for  KR solitons in  the $T$-equivalent deformation space.

   \begin{cor}\label{ino-theorem} Let  $(M,  J_0,  \omega)$ be a Fano manifold which  admits  a KR soliton   $(\omega_{KS}, X)$.
Then there exists a small $\epsilon$-ball $B_\epsilon(0)\subset H_{T}^1(M,J_0,\Theta)$ such that
   for any $\psi_\tau\in B_\epsilon(0)$ the Fano manifold $(M, J_{\psi_\tau})$  admits a KR soliton if and only if $(M, J_{\psi_\tau})$ is modified  $K$-polystable.
 \end{cor}

 \begin{rem}\label{inoue-Th} By using the deformation theory,  Inoue proved the sufficient part of Corollary \ref{ino-theorem} in sense of the GIT-polystability via the group ${\rm Aut}(M, J_0)$ \cite[Proposition 3.8]{Ino19}.  But it is still unknown  whether the necessary part is true in sense of the GIT-polystability  \cite[Postscript Remark 1]{Ino19}.

 \end{rem}

 Next we prove the following uniqueness result  for  KR-solitons in the closure of the orbit by diffeomorphisms.

   \begin{theo}\label{uniqueness-orbit}  Let $\{\omega_i^1\}$  and $\{\omega_i^2\}$ be two  sequences of K\"ahler metrics in $2\pi c_1(M, J)$ which converge to  KR-solitons  $(M_\infty^1,\omega_{KS}^1)$ and  $(M_\infty^2, \omega_{KS}^2)$ in sense of Cheeger-Gromov, respectively. Suppose that
   \begin{align}\label{max-entropy-10}\lambda(\omega_{KS}^1)=\lambda(\omega_{KS}^2)=\sup\{\lambda(\omega_{g'})|~\omega_{g'}\in 2\pi c_1(M,J)\}.
  \end{align}
   Then  $M_\infty^1$ is biholomorphic to $M_\infty^2$ and  $\omega_{KS}^1$ is isometric to $\omega_{KS}^2$.

   \end{theo}

  Theorem \ref{uniqueness-orbit} generalizes   the uniqueness result of Tian-Zhu for  KR-solitons \cite{TZ00}  as well as a  recent result of Wang-Zhu  \cite[Theorem 0.4]{WZ1} in sense of diffeomorphisms orbit  where  both of ${\rm Aut}_0( M_\infty^1)$ and  ${\rm Aut}_0( M_\infty^2)$ are assumed to be   reductive\footnote{In fact, we expect to generalize Theorem \ref{uniqueness-orbit} for the singular KR-solitons of  $(M_\infty^1,\omega_{KR}^1)$ and  $(M_\infty^2, \omega_{KR}^2)$ as in \cite[Theorem 6.7]{WZ1}},  and  it  is  also  a generalization of  uniqueness result of Chen-Sun for  KE-metrics orbit  \cite{CS14} (also see \cite{LXW18}).  We also note  that the assumption (\ref{max-entropy-10})  is necessary  according to Pasquier's counter-example  of  horospherical variety  \cite [Remark 6.5]{WZ1}.

  Recall \cite{WZ20}

  \begin{defi}\label{deformation-j}Let $(M,J)$ be a Fano manifold. A complex manifold  $(M', J')$  is called a canonical smooth deformation of  $(M,J)$
  if there are a sequence of K\"ahler metrics $\omega_i$ in $2\pi c_1(M,J)$
  and  diffeomorphisms $\Psi_i: M'\to M$ such that
  \begin{align}\label{c-infty-convergence-vector}\Psi_i^* \omega_i     \stackrel{C^\infty}{\longrightarrow} \omega', ~ \Psi_i^* J \stackrel{C^\infty}{\longrightarrow} J', ~{\rm on}~ M'.
  \end{align}
  In addition that  $J'$ is not conjugate to $J$,  $J'$  is called a jump of $J$.

  \end{defi}

 Theorem \ref{uniqueness-orbit} is a direct corollary of   following convergence result of KR flow together with the uniqueness result  of  Han-Li \cite{HL} (cf.  \cite{WZ1}).

  \begin{prop}\label{stability-KR-smooth-complex} Let $(M', J')$  be  a canonical smooth jump of a Fano manifold  $(M,J)$.  Suppose that   $(M', J')$ admits a KR soliton $\omega_{KS}$ such that
  \begin{align}\label{max-entropy-4}\lambda(\omega_{KS})=\sup\{\lambda(g)|~\omega_{g}\in 2\pi c_1(M,J)\}.
  \end{align}
  Then   for any initial metric  $\widetilde\omega_0\in 2\pi c_1(M,J)$
  the  flow $(M, J, \widetilde\omega(t))$ of   (\ref{kr-flow}) is  uniformly  $C^\infty$-convergent to    $(M', J',\omega_{KS})$.
  \end{prop}

 \begin{proof} By the assumption, there is a sequence of K\"ahler metrics   $\omega_i$  in $2\pi c_1(M,J)$  such that
  $$\lim_i{\rm dist}_{CG}((M,\omega_i), (M', \omega'))=0,$$
where $\omega'\in 2\pi c_1(M', J')$.
Note that the  KR flow of (\ref{kr-flow}) with the initial $\omega'$ converges to the KR soliton $\omega_{KS}$ \cite{TZZZ, DeS}.
  Then by the  stability of  (\ref {kr-flow})  for the finite time,  there is  a sequence of $\tilde \omega_i^{t_i}$ such that
  $$\tilde \omega_i^{t_i}=\omega_{KS}+\sqrt{-1}\partial_{J}\bar\partial_{J}\chi_i$$
  and
   $$\lim_i{\rm dist}_{CG, C^3}((M, \tilde\omega_i^{t_i}), (M', \omega_{KS}))=0,$$
   where $\tilde\omega_i^{t}$ is the solution of flow  (\ref{kr-flow}) with the initial metric $\omega_i$. It follows that there are diffeomorphisms $\Phi_i$ such that
   $$\lim_i\|\Phi_i^*J-J'\|_{C^{4,\gamma}}=0. $$
   Thus there are diffeomorphisms $K_i$ (cf.  (\ref{k-difformation})) such that
   $$J_{\varphi_i}=K_i^*(\Phi_i^*J)~{\rm and}~
   \|\varphi_i\|_{C^{k+4,\gamma}(M)}<<1, ~{\rm as}~i>>1. $$
    Hence,   as in the proof of Theorem \ref{main-theorem-2} \footnote{(\ref{lower-bound-2}) always holds by (\ref{max-entropy-4}).},   we  conclude that the sequence  $\{g_i(t)\}$ of  modified  flows of  (\ref{normalization-g}) uniformly converges to the soliton flow generated by $\omega_{KS}$. Since the  limit of $g_i(t)$ is independent of the initial metrics $\tilde\omega_i^{t_i}$  \cite{HL, WZ1},  each  $g_i(t)$ must converge to  $\omega_{KS}$.  Again by  the uniqueness result in   \cite{HL, WZ1},  we prove the proposition.

 \end{proof}

By Proposition \ref{stability-KR-smooth-complex}, we in particular prove  the following  uniqueness  result for the centers  of smooth degenerations with admitting  KR solitons.

 \begin{cor}\label{uniqueness-orbit-2} Let $(M,J)$ be a  Fano manifold. Suppose  there is a   smooth degeneration  on $(M,J)$ with   its center  $(M',J')$ admitting  a KR soliton $g_{KS}$ which satisfies
   \begin{align}\label{max-entropy-degeneration}\lambda(g_{KS})=\sup\{\lambda(g)|~\omega_{g}\in 2\pi c_1(M,J)\}.
    \end{align}
   Then  $(M',J')$ is unique.

\begin{proof} By  Corollary \ref{kur} and  Kuranishi's  completeness  theorem   \cite{MK}, there is a smooth path $\tau(s)\in B(\epsilon)$ ~$(s<<1)$ such that $g_0=g_{KS}$. Then
$$g_{\tau(s)}={\omega}(\cdot, J_{\tau(s)}\cdot)$$
induces   a   smooth jump of  $(M,J)$ to  $(M',J')$. Here $\omega$ is the K\"ahler form of $g_{KS}$. Thus   the corollary comes from Proposition \ref{stability-KR-smooth-complex} immediately.

\end{proof}

   \end{cor}

\section{Generalization of Theorem \ref{main-theorem-3}}

In this section, we generalize the argument in the proof of Theorem \ref{main-theorem-3} to prove Theorem \ref{main-theorem-tivial-action}.  We will use a different  way to obtain  the  entropy estimate (\ref{lower-bound-2}) in the proof of  Theorem \ref{main-theorem-2}.  We first
 consider  the special case of  $H_1< 0$, i.e., the stable case,   and prove the jump of   KR flow (\ref{kr-flow}) to the KR soliton $(M, J_0, \omega_{KS})$.

 \subsection{Linearly stable case}

 \begin{theo}\label{main-theorem-jump}  Let  $(M,  J_0)$ be a Fano manifold  which admits  a KR soliton   $\omega_{KS}\in 2\pi c_1(M, J_0)$. Suppose that for  any $\psi\neq 0~{\rm in}~\mathcal{H}^{0, 1}_{h}(M,T^{1, 0}M)$ it holds
 \begin{align}\label{non-positive-2}
 H_1(\psi)< 0.
 \end{align}
  Then there exists  a small  $\epsilon$ such that
 for any $\tau\in B(\epsilon) $  the flow (\ref{normalization-g}) with an initial metric $g=L(\tau,\chi)$
  converges smoothly  to  the KR soliton $(M, J_0,  \omega_{KS})$. Furthermore, the convergence is fast  in the polynomial rate.
 \end{theo}

 \begin{proof}  By  the Hamilton-Tian conjecture and the uniqueness result in   \cite{HL, WZ1}, we need to prove the convergence of  flow  $g(t)$ of (\ref{normalization-g})  with the initial metric $g_\tau=L(\tau, 0)$ for any $\tau\in  B(\epsilon)$.  Moreover by (\ref{non-positive-2}),  we may assume that
 there is $\psi\in\mathcal{H}^{0, 1}_{\theta}(M,T^{1, 0}M) $ such  that
 \begin{align}\label{normalize-psi-4}\|\psi\|=1, \psi_\tau=\sum\tau_ie_i= \tau\psi=\tau(\psi_1+ \psi')
 \end{align}
 with the property:
 \begin{align}\label{normalize-psi-5}& \|\psi_1\|\ge \frac{1}{l}>0,  H_1(\psi_1)= -\delta_1\psi_1;
 \notag\\
 &    <H_1(\psi'),  \psi'>\le 0, ~\|\psi'\|\le 1.
 \end{align}
 Here $\delta_1\ge 0$.  Without loss of generality, $-\delta_1$   may be regarded as the largest negative  eigenvalue of $H_1$.
 Thus   for any $\chi$ with $\|\chi\|_{C^{4,\gamma}}\le \tau N_0$, where $N_0$ is fixed,  the following is true:
 \begin{align}\label{taylor-ex-3}&(\psi_\tau, \chi)=\tau(\psi, \chi'), ~\|\chi'\|_{C^{4,\gamma}}\le N_0,\notag\\
 &\lambda(g_{\tau,\chi})\le\lambda(g_{KS})-\tau^2(\frac{\delta_1}{l^2}+o(1))<\lambda(g_{KS}).
 \end{align}

 Let $T=T_{\tau, 0}$  be  the set  defined as  in  (\ref{T-set}) with $\delta_0=O(\tau)$.  Let  $\tilde\psi=\psi_t\in \mathcal{H}^{0,1}_{\theta}(M,T^{1,0}M)$ with $\|\tilde\psi\|<\epsilon_0=O(\tau)$ defined as in (\ref{k-difformation}) such that
 $$K(t)^*J(t)=J_{\tilde\varphi}, ~\tilde\varphi=\Phi(\tilde\psi).$$
  Then as in (\ref{normalize-psi-4}) and  (\ref{normalize-psi-5}), there is
 $\tilde\psi_1\in \mathcal{H}^{0,1}_{\theta}(M,T^{1,0}M)$ with $\|\tilde \psi_1\|\ge \frac{1}{l}$
 such that
 \begin{align}\label{phi-t-orthogonal}& \|\tilde\psi\|=\tau'<\epsilon_0;\notag\\
 & \tilde\psi=  \tau'(\tilde\psi_{1}+ \tilde\psi'),  \|\tilde\psi_{1}+ \tilde\psi'\|=1; \notag\\
 &  H_1(\tilde\psi_{1}) =-\delta_1' \tilde\psi_{1} ~{\rm and}~  <H_1(\tilde\psi'),  \tilde\psi'>\le 0,
 \end{align}
 where $\delta_1'\ge\delta_1.$

 On the other hand,  we may assume that $\omega_{g(t)}$ is same as $\omega_{KS}$. Otherwise, by Moser's theorem, there is a diffeomorphism  $\hat K(t)\in {\rm Diff(M)}$ such that
 $\omega_{\hat g(t)}=\omega_{KS},$
 where $\hat g(t)=\hat K(t)^*g(t)$. Then we can replace $g(t)$ by  $\hat g(t)$.
Thus,  there are $K(t)\in {\rm Diff(M)}$ and
  $\tilde\chi\in C^\infty(M)$ as in (\ref{potential}) such that
 \begin{align}\label{chi-equation-2}\omega_{ K(t)^*g(t)}==K(t)^*\omega_{KS}=\omega_{KS}+\sqrt{-1}\partial_{J_{\tilde\varphi}}\bar\partial_{J_{\tilde\varphi}}
 \tilde\chi,~\int_M \tilde\chi\omega_{KS}^n=0.
 \end{align}
 Moreover, as in  (\ref{metric-comparison}), we have
 $$K(t)^*\omega_{KS}=(1+O(\|\tilde\varphi_t\|_{L^2(M)}))\omega_{KS}.$$
 It follows   that (cf. (\ref{chi-small-improve}))
 $$  \|\Delta\tilde\chi\|_{C^{4,\gamma}}\le \tau' N_0.$$
 Consequently,  by (\ref{taylor-ex-3}), we derive
 \begin{align}\label{lmabda-decay-4}\lambda(g(t))=\lambda( L{\tilde\psi,\tilde\chi})\le\lambda(g_{KS})-(\tau')^2(\frac{\delta_1}{l^2}+o(1))\le\lambda(g_{KS}).
 \end{align}
 Hence, we can continue the argument in the proof of Theorem \ref{main-theorem-2} to show that
  the flow (\ref{normalization-g})
  converges smoothly  to  a KR soliton $(M, J_\infty,  \omega_{\infty})$ such that
   $$\lambda(g_{\infty})=\lambda(g_{KS}).$$
 Moreover, by (\ref{lmabda-decay-4}), $\tau'$ must go to zero.  Therefore,  $\tilde\psi$ goes to zero and  $g_{\infty}$ is same as $g_{KS}$  by the uniqueness of KR solitons \cite{TZ00}.

 For a general initial metric in $2\pi c_1(M,  J_\tau)$, we only   remark how to get the decay estimate. In fact,
 we can choose a large $T_0$ such that
 $${\rm dist}_{CG, C^k}((M, g(T_0)), (M,  g_\infty))\le \epsilon$$
 and  there is an auto-diffeomorphism $\sigma$ such that
 $$\|\sigma^*J({T_0})- J_0\|_{C^k}\le \epsilon'.$$
 Then as in (\ref{chi-equation-2}), there is  an auto-diffeomorphism  $K(T_0)$  and    $\psi\in\mathcal{H}^{0, 1}_{\theta}(M,T^{1, 0}M)$ with $\varphi=\Phi'(\psi)$   for the metric $\sigma^*g(T_0)$ such that
 $$\|\psi\|=\tau<<1~ {\rm and}  \|\Delta\chi\|_{C^{4,\gamma}}\le \tau N_0.$$
 Hence, by the above argument, the KR flow with the initial metric  $g(T_0)$ converges to  $g_{\infty}$  fast in the polynomial rate. The proof is complete.

 \end{proof}

Theorem \ref{main-theorem-jump} means that any KR flow with an initial metric in $2\pi c_1(M, J_{\tau})$ with $\tau<<1$ under the condition
(\ref{deformation-j}) will jump upto the original KR soliton $(M, J_0, \omega_{KS})$ and also show that the KR soliton $(M, J_0, \omega_{KS})$ is isolated in the deformation space of complex structures in this case. But upto now, we do not know whether there is such an example of Fano manifold with admitting  an isolated   KR soliton around her deformation space of complex structures.

\subsection{Linearly semistable  case}

 In this subsection,  we generalize  Theorem  \ref{main-theorem-jump} with the existence of   non-trivial     kernel $Z$   of operator $H_1$ to prove  Theorem \ref{main-theorem-tivial-action}.
Let $K$ be a maximal compact subgroup  of ${\rm Aut}_r(M, J_0)$ which contains  $K_X$.
 Then $(\omega_{KS}, J_0)$ is $K$-invariant \cite{TZ00}, and so  $\sigma$ maps   $\mathcal{H}^{0, 1}_{\theta}(M,T^{1, 0}M)$ to itself  for any $\sigma\in K$.
 The latter  means that $\sigma(\psi)$ is still  $\theta$-harmonic whenever   $\psi$ is $\theta$-harmonic. Thus  we can extend the  action  ${\rm Aut}_r(M, J_0)$ on  $\mathcal{H}^{0, 1}_{\theta}(M,T^{1, 0}M)$ so that  $\mathcal{H}^{0, 1}_{\theta}(M,T^{1, 0}M)$ is an ${\rm Aut}_r(M, J_0)$-invariant space. Since
$$K_X\cdot \sigma= \sigma\cdot K_X,~\forall \sigma \in {\rm Aut}_r(M, J_0), $$
we have
$$\mathcal L_\xi \sigma^*\psi=0, ~\forall ~\psi\in Z.$$
  Hence, $Z$ is also an  ${\rm Aut}_r(M, J_0)$-invariant subspace of $\mathcal{H}^{0, 1}_{\theta}(M,T^{1, 0}M)$. Furthermore, we prove

\begin{lem}\label{trivial-action-1}
 Suppose that  $(M,J_{\psi_\tau})$ admits a KR soliton for any  $\psi_\tau\in B(\epsilon)\cap Z$.  Then
      \begin{align}\label{isotropic}\sigma^*{\psi}=\psi_, ~ \forall~\sigma\in {\rm Aut}_r(M, J_0), \psi \in Z.
      \end{align}

\end{lem}

\begin{proof}
Since the union of maximal tori  is dense  in a reductive  complex group,  we need to  prove  (\ref{isotropic}) for any  torus 1-ps  $\sigma_t$ and $\psi\in Z$.    On the contrary,  there are   $\sigma\in  {\rm Aut}_r(M, J_0)$ and  $\psi\in Z$  such that $\sigma_t=\sigma^{-1}\cdot\rho_t\cdot\sigma$  with some nontrivial weights for $\psi$,
 where
   $$\rho_t:\mathbb{C}^*\mapsto   {\rm Aut}_r(M, J_0) $$
   is  a diagonal  torus action on  $Z$.  Since $Z$ is   ${\rm Aut}_r(M, J_0)$-invariant, there is a   basis of  $Z$,
 $\{e_1,\ldots,e_s\}$ $( s={\rm dim} Z)$ such that $\sigma(\psi)=\sum_{i\le s} e_i$ and
$$\sigma_t (\psi)= \sum_i t^{a_i} \sigma^{-1}(e_i), t\in \mathbb{C}^*,$$
    where $a_i\in \mathbb{Z}$ are  weights of  $\sigma_t$ with some $a_i\neq0$.
      Without loss of generality,  we may assume  $a_1>0$.

Now we consider  a degeneration of $\rho_t(e_1)$ $(|t|<\epsilon, |e_1|<<1)$ in $Z$. Note that
$$J_{\rho_t^*\Phi( e_1)}=\rho_t^* J_{\Phi( e_1)}.$$
Then by the Kuranishi theorem, there is a $K_t\in {\rm Diff(M)}$ such that
$$J_{\Phi(\rho_t (e_1))}= K_t^*(\rho_t^* J_{\Phi( e_1)}).$$
Thus, for any small  $t_1$ and $t_2$ it holds
  $$(M, J_{\Phi(t_1e_1)})\cong (M, J_{\Phi(t_2e_1)}).$$
Hence,  the family of $(M, J_{\Phi(te_1)})$ $(|t|<\epsilon)$ is generated by a
 $C^*$-action  on $(M, J_{\Phi(e_1)})$,  which degenerates to  $(M, J_0)$.  Moreover, it communicates  with 1-ps $\sigma_s^X$ and its modified Futaki-invariant is zero.  Therefore,  we get  a contradiction with the modified   K-polystability of  $(M, J_{\Phi(e_1)})$ which admits a KR soliton. The proof is finished.

\end{proof}

\begin{rem}In case of  $\omega_{KS}=\omega_{KE}$,  $Z=\mathcal{H}^{0, 1}(M,T^{1, 0}M)$. Then the property
(\ref{isotropic})  turns to hold for any  $\sigma\in {\rm Aut}(M, J_0), \psi \in \mathcal{H}^{0, 1}(M,T^{1, 0}M)$.  Thus Lemma  \ref{trivial-action-1} is actually a generalization of  \cite[Theorem 3.1]{CSYZ} in  the   case
 of KE metrics.
 \end{rem}

By Lemma \ref{trivial-action-1}, we prove

\begin{lem}\label{trivial-maximality} Under the assumption in  Lemma \ref{trivial-action-1},
 there exists $\delta>0$ such that for any
$(\tau,\chi)\in \mathcal{V}_{\delta}$ it holds
$$\lambda(g_{\tau, \chi})\leq \lambda(g_{KS}), $$
 if
 $H_1(\cdot)\leq 0$ on $\mathcal{H}^{0, 1}_{\theta}(M,T^{1, 0}M).$

\end{lem}

\begin{proof}
Let  $ \mathfrak{k}$ be the Lie algebra of $K$, and $\eta_u$ the nilpotent part of $\eta(M,J_0; \mathbb R)$.  Then
$$\eta(M,J_0; \mathbb R)=\eta_r\oplus \eta_u=\mathfrak{k}\oplus J_0\mathfrak{k}\oplus \eta_u.$$
Fix a small ball $U(\epsilon_0)\subseteq \mathfrak{k}\oplus \eta_u$ centered at the original.
By the Kuranishi theorem  there exists $\tau'\in B(\epsilon)$ and $\gamma\in {\rm Diff}(M) $ for any $\tau\in B(\epsilon), (v_1,v_2)\in U(\epsilon_0)$ such that
\begin{align}\label{Kur-decom}
(e^{J_{\tau}v_1+v_2})^*J_{\tau}=\gamma^*J_{\tau'},
\end{align}
where $J_{\tau}$ and $J_{\tau'}$ are complex structures associated to $\psi_\tau$ and  $\psi_{\tau'}$, respectively.    We define  a  smooth map
$\Phi:U(\epsilon)\times \mathcal{V}_{\epsilon}\mapsto \mathcal{V}_{\epsilon}$ by
$$\Phi((v_1,v_2),(\tau,\chi))=(\tau',\chi'),$$
where the K\"ahler potential $\chi'$ is determined by
\begin{align}\label{sigma-tran}(\gamma^{-1})^*(e^{J_{\tau}v_1+v_2})^*(\omega+\sqrt{-1}\partial_{J_{\tau}}\bar\partial_{J_{\tau}}\chi)
=\omega+\sqrt{-1}\partial_{J_{\tau'}}\bar\partial_{J_{\tau'}}\chi' ~{\rm and}~\int_M\chi'\omega_{KS}=0.
\end{align}

Let
$\Psi:U(\epsilon)\times \mathcal{V}_{\epsilon}\mapsto {\rm ker}(H_2)$ be a projection of $\Phi$
such that
$$\Psi={\rm Pr}_{{\rm ker}(H_2)}\circ {\rm Pr}_2\circ \Phi.$$
Namely,
$$\Psi(((v_1,v_2)(\tau,\chi)))={\rm Pr}_{{\rm ker}(H_2)}(\chi').$$
A direct calculation  shows  that
$$D_{(0,0)}\Psi((v_1,v_2),0)=\theta_{Y}+\overline{\theta_{Y}},$$
where $\theta_Y$ is a potential of HVF,
$$Y=\frac12\sqrt{-1}(v_1-\sqrt{-1}J_0v_1)+\frac12(v_2-\sqrt{-1}J_0v_2).$$
Thus
by \cite{TZ08},  we see that
$$D_{(0,0)}\Psi(\cdot,0):\mathfrak{k}\oplus \eta_u\mapsto {\rm ker}(H_2)$$
is an isomorphism.  By the implicit  function theorem,   there are  $\delta<\epsilon$  and  map $G: \mathcal{V}_{\delta} \to U(\varepsilon_0)$ such that
\begin{align}{\rm Pr}_2\circ \Phi(G(\tau,\chi),(\tau,\chi))\in {\rm ker}(H_2)^{\perp}.\notag
\end{align}
It follows that
\begin{align}\label{projection-map} (\tau',\chi')=\Phi(G(\tau,\chi),(\tau,\chi))\in  \mathcal{V}_{\delta'},
\end{align}
where $\chi'\in  {\rm ker}(H_2)^{\perp}$ and  $\delta'(\delta)<<1.$

Let
$$\tilde{\mathcal{V}}_{\epsilon}=\{(\tau_0,\widetilde{\chi})\in \mathcal{V}_{\epsilon}|~\tau_0\in Z,~\xi(\widetilde{\chi})=0\}.$$
Then the  restricted map $\widetilde{\Psi}$ of   $\Psi$ on $\mathfrak{k}\times \tilde{\mathcal{V}}_{\epsilon}$  is also smooth and
$$D_{(0,0)}\widetilde{\Psi}(\cdot,0):\mathfrak{k}\mapsto
\{\theta_{Y}+\overline{\theta_{Y}} \} $$
is  an isomorphic map.
Here $Y=\frac12\sqrt{-1}(v-\sqrt{-1}J_0v),  ~v\in\mathfrak{k}.$ We claim:
\begin{align}\label{orthigonal}{\rm ker}(H_2)\cap\{\chi|~\xi(\chi)=0\}=\{\theta_{Y}+\overline{\theta_{Y}}\}.
\end{align}

On contrary, there is a nontrivial  $v'~\in\eta_u$ such that
$$\xi(\theta_{Y'}+ \overline{\theta_{Y'}})= 0, $$
where $Y'=v'-\sqrt{-1}J_0v'$.  Then for any $\sigma\in K_X$,
$$\sigma^*(\theta_{Y'}+ \overline{\theta_{Y'}})=\theta_{Y'}+ \overline{\theta_{Y'}}.$$
It follows that
$$\sigma^*Y'=Y', ~\forall~\sigma\in K_X.$$
This implies $Y'\in \eta_r(M, J_0)$, which is a contradiction!

 By (\ref{orthigonal}),
 $$D_{(0,0)}\widetilde{\Psi}(\cdot,0):\mathfrak{k}\mapsto {\rm ker}(H_2)\cap\{\chi|~\xi(\chi)=0\}$$
is  an isomorphism.  On the other hand,  by  Lemma \ref{trivial-maximality},
$$(e^{J_{\tau_0}v})^*J_{\tau_0}=J_{\tau_0}, ~ \forall~ \psi_{\tau_0}\in Z, v\in \mathfrak{k}.$$
 Hence, applying the implicit  function theorem to $\widetilde{\Psi}$ at $(\tau_0,\widetilde{\chi})\in \tilde{\mathcal{V}}_{\epsilon}$, there are $v\in\mathfrak{k}$ and  K\"ahler potential  $\chi_0\in {\rm ker}(H_2)^{\perp}$ such that $(\tau_0,\chi_0)\in \tilde{\mathcal{V}}_{\delta'}$ and
$$(e^{J_{\tau}v})^*(\omega+\sqrt{-1}\partial_{J_{\tau_0}}\bar\partial_{J_{\tau_0}}\widetilde\chi)
=\omega+\sqrt{-1}\partial_{J_{\tau_0}}\bar\partial_{J{\tau_0}}\chi_0.
$$
 Note that  a K\"ahler potential $\widetilde\chi$  of KR soliton with respect to $X$ on $(M, J_{\tau_0})$ satisfies $\xi(\widetilde\chi)=0$ \cite{TZ00}. Therefore,   there is a KR soliton $\omega+\sqrt{-1}\partial_{J_{\tau_0}}\bar\partial_{J{\tau_0}}\chi_0$   on $(M, J_{\tau_0})$ such that
\begin{align}\label{soliton-potential}\chi_0\in {\rm ker}(H_2)^{\perp} ~{\rm and}~(\tau_0,\chi_0)\in \tilde{\mathcal{V}}_{\delta'}.
\end{align}

Decompose  $\psi_{\tau'}$ in (\ref{projection-map}) as
\begin{align}\label{decomposition-1}\psi_{\tau'}=\psi_{\tau_0}+\psi_{\tau_1},~\psi_{\tau_0}\in Z,~\psi_{\tau_1}\in Z^{\perp}.
\end{align}
Then by (\ref{soliton-potential}), there are K\"ahler potentials  $\chi_0$ and $\chi_1\in {\rm ker}(H_2)^{\perp}$ such that
\begin{align}\label{decomposition-2}\chi'=\chi_0+\chi_1,
\end{align}
where $\omega+\sqrt{-1}\partial_{J_{\tau_0}}\bar\partial_{J{\tau_0}}\chi_0$ is   a KR soliton.
 Thus, we   get  the  following expansion for the reduced entropy $\nu$,
\begin{align}\label{expansion-lamda}
  \nu(\tau',\chi')=\nu(\tau_0,\chi_0)+\frac12(D^2_{(\tau_0,\chi_0)+\zeta(\tau_1,\chi_1)}\nu)((\tau_1,\chi_1),(\tau_1,\chi_1)).
\end{align}
Here $\zeta\in (0,1)$.  Since $D^2_{(0,0)}\nu (0, \cdot)$ is strictly negative on ${\rm ker}(H_2)^{\perp}$ \cite{TZ08}, by the assumption  in Lemma \ref{trivial-maximality},  it is strictly negative on ${\rm ker}(H_1)^{\perp}\times {\rm ker}(H_2)^{\perp}$. It follows that $D^2_{(\tau_0,\chi_0)+\zeta(\tau_1,\chi_1)}\nu$  is strictly negative on ${\rm ker}(H_1)^{\perp}\times {\rm ker}(H_2)^{\perp}$.  Hence,  we derive
\begin{align}\label{maximal-2}
  \nu(\tau',\chi')\leq\nu(\tau_0,\chi_0).
\end{align}
As a consequence, we prove
$$\lambda(g_{\tau,\chi})\leq \lambda(g_{\tau_0,\chi_0})=\lambda(g_{KS}).$$
\end{proof}

By Lemma  \ref{trivial-maximality}, we can use the argument  in the proof of  Theorem \ref{main-theorem-2} (or Theorem \ref{main-theorem-jump}) to  finish the proof of  Theorem \ref{main-theorem-tivial-action}. We leave it to the reader.

We end this subsection by the following two remarks to  Theorem \ref{main-theorem-tivial-action}.
\begin{rem}\label{remark1-thm2}
 In Theorem \ref{main-theorem-tivial-action}, we can actually prove that the   convergence is of    polynomial rate of any order. This is because we can improve  the  Lojasiewicz  inequality of (\ref{loja}) for $\nu(\cdot)$  in Proposition \ref{lojasiewsz} as follows:
 For any $a=(\varphi, \chi)\in\mathcal V_\epsilon $  it holds
 \begin{align}\label{loja-2}
 \|\nabla\nu(a')\|_{L^2}\geq c_0|\nu(a')-\nu(0)|^{\frac{1}{2}},
 \end{align}
 where $a'= (\tau',\chi')\in  \mathcal{V}_{\delta'}$ is determined in (\ref{projection-map}).
It follows that  the  Lojasiewicz  inequality of (\ref{lajw-inequ-3})  for $\lambda(\cdot)$ along the flow  holds for any $\alpha'> \frac{1}{2}.$
  Hence, by  (\ref{decay-g}),  the   convergence is of  polynomial rate  of any order. (\ref{loja-2}) can be proved together with  (\ref{nabal-nu2}) and (\ref{nu-0}) since the functional  $F$ there can vanish by the decompositions (\ref{decomposition-1}) and (\ref{decomposition-2}).
\end{rem}

\begin{rem}\label{remark2-thm2}
 We call the KR flow $\widetilde\omega(t)$ of (\ref{kr-flow})  is stable  on $(M, J_{\psi_\tau})$ if it converges smoothly to a KR soliton $(M_\infty, J_\infty, \omega_{\infty})$ with $\lambda(\omega_\infty)= \lambda(\omega_{KS})$.
  In   Theorem \ref{main-theorem-tivial-action},    we show that  the  flow (\ref{kr-flow}) is always stable for any complex structure in the deformation space under the conditions of (1) and (2).
 By the above argument  (also see Remark \ref{DS-result}),  it is easy to see that the flow $\widetilde\omega(t)$   is stable on $(M, J_{\psi_\tau})$ if and only if
 the energy level  $L([M, J_{\psi_\tau}])$ of flow is equal to $\lambda(\omega_{KS}).$ Here
   \begin{align}\label{level-energy} L([M, J_{\psi_\tau}])=\lim_t\lambda(\widetilde\omega(t)),
   \end{align}
   which is independent of flow  $\widetilde\omega(t)$  with the initial metric $ \widetilde\omega_0 \in 2\pi c_1(M,  J_{\psi_\tau})$
    \cite{ZhT, DeS, WZ1}.
\end{rem}


   \end{document}